\documentclass[oneside,english]{amsart}
\usepackage[T1]{fontenc}
\usepackage[latin9]{inputenc}
\usepackage{prettyref}
\usepackage{amsthm}
\usepackage{amssymb}

\makeatletter
\numberwithin{equation}{section}
\numberwithin{figure}{section}
\theoremstyle{plain}
\newtheorem{thm}{\protect\theoremname}
  \theoremstyle{plain}
  \newtheorem{lem}[thm]{\protect\lemmaname}
  \theoremstyle{definition}
  \newtheorem{defn}[thm]{\protect\definitionname}
  \theoremstyle{remark}
  \newtheorem{rem}[thm]{\protect\remarkname}
  \theoremstyle{plain}
  \newtheorem{prop}[thm]{\protect\propositionname}
  \theoremstyle{remark}
  \newtheorem{claim}[thm]{\protect\claimname}

\makeatother

\usepackage{babel}
  \providecommand{\claimname}{Claim}
  \providecommand{\definitionname}{Definition}
  \providecommand{\lemmaname}{Lemma}
  \providecommand{\propositionname}{Proposition}
  \providecommand{\remarkname}{Remark}
\providecommand{\theoremname}{Theorem}

\begin{document}

\title{Injectivity of a certain cycle map for finite dimensional W-algebras}

\author{Christopher dodd}
\begin{abstract}
We study a certain cycle map defined on finite dimensional modules
for the W-algebra with regular integral central character. Via comparison
with the theory in postive characteristic, we show that this map injects
into the top Borel-Moore homology group of a Springer fibre. This
is the first result in a larger program to completely desribe the
finite dimensional modules for the W algebras.
\end{abstract}
\maketitle
\tableofcontents{}

\section{Introduction}

Recently, the subject of the finite $W$-algebras has come to the
attention of many researchers. Although originally introduced in the
physics literature, they were first defined in a mathematical context
by Premet \cite{key-21}, who related them to the study of (non-restricted)
modular representations of semisimple lie algebras. The fundamental
paper of Gan and Ginzburg \cite{key-13} reproved some of Premet's
results, and recast them in the light of non-commutative algebraic
geometry. Since then, many authors have made contributions to their
study, c.f., e.g., \cite{key-22,key-18,key-19,key-8,key-11}, and
the survey articles \cite{key-30} and \cite{key-31} have appeared.
In particular, the results of \cite{key-8} and \cite{key-19} are
concerned with the finite dimensional representations of $W$-algebras.
Despite the significant progress made there, some fundamental questions
remain open. For instance, given a finite $W$-algebra $U(\mathfrak{g},e)$
and an integral central character $\lambda$, it is still not known
how to parametrize the simple finite dimensional $U(\mathfrak{g},e)$
modules with character $\lambda$. The goal of this paper is to provide
the first step to answering this question. In fact, we will provide
some detailed information on the $K$-group $K_{\mathbb{Q}}(mod^{f.d.}(U^{\lambda}(\mathfrak{g},e)))$. 

Our main tool will be the use of a certain characteristic cycle map
which takes $K_{\mathbb{Q}}(U^{\lambda}(\mathfrak{g},e))$ to the
the homology group $H_{top}(\mathcal{B}_{e},\mathbb{Q})$- this is
the top Borel-Moore homology of the Springer fibre associated to the
nilpotent element $e$ (definitions will be recalled below). The latter
group has a natural basis (as a $\mathbb{Q}$- vector space) indexed
by irreducible components of the variety $\mathcal{B}_{e}$. In addition,
it has the structure of a module over the Weyl group associated to
$\mathfrak{g}$, called $W$. This is the classical construction of
Springer, which finds all of the simple $W$-modules in such homology
groups. 

The group $K_{\mathbb{Q}}(U^{\lambda}(\mathfrak{g},e))$ also has
a natural structure of a $W$-module, via the action of reflection
functors on the category $U^{\lambda}(\mathfrak{g},e)-mod$. The theory
of these functors, which is parallel to the classical theory of reflection
functors for $U(\mathfrak{g})$ developed by Jantzen, has been worked
out in \cite{key-35}. We shall recall their basic properties below. 

With all of this in hand, we can state the basic theorem of this paper. 
\begin{thm}
The cycle map 
\[
cc:K_{\mathbb{Q}}(mod^{f.d.}(U^{\lambda}(\mathfrak{g},e)))\to H_{top}(\mathcal{B}_{e},\mathbb{Q})
\]
 is injective and $W$-equivariant, with respect to the actions of
$W$ discussed above. 
\end{thm}
Let us note right away that the paper \cite{key-12} provides a weaker
version of this result; namely, they prove the numberical bound 
\[
\mbox{dim}(K_{\mathbb{Q}}(mod^{f.d.}(U^{\lambda}(\mathfrak{g},e))))\leq\mbox{dim}(H_{top}(\mathcal{B}_{e},\mathbb{Q}))
\]

Their proof uses $D$-module theory in characteristic zero and the
theory of springer's representations.

The main tool of this paper will be the use of reduction mod $p$
and a comparison with the theory of $W$-algebras over algebraically
closed fields of positive characteristic. Following the reasoning
of \cite{key-5}, we shall construct a localization theory for such
algebras, and compare the resulting geometry with the geometry in
characteristic zero. In positive characteristic, we have extra tools
such as the Azumaya splitting and the action of the frobenius morphism.
We shall use these to deduce the result for sufficiently large positive
characteristic, and then we shall transfer to characteristic zero.

We should also point out that I. Losev and V. Ostrik have a complete
description of the image of the map $cc$, which will appear in a
forthcoming work. To state it, let us recall that to the nilpotent
element $e$, we can also associate a (possibly trivial) cell $c$
in the Weyl group $W$ (c.f. \cite{key-26} for a complete introduction
to cells and representations of Weyl groups and Hecke algebras). To
this cell we can then associate 
\[
H_{top}(\mathcal{B}_{e},\mathbb{Q})_{c}
\]
 a sub-$W$-representation of $H_{top}(\mathcal{B}_{e},\mathbb{Q})$.
They conjecture that this is the image of the map $cc$. Therefore,
combined with the result in this paper, this theorem yields a complete
description of $K_{\mathbb{Q}}(mod^{f.d.}(U^{\lambda}(\mathfrak{g},e)))$.

\thanks{The author would like to express his gratitude to Ivan Losev, for
suggesting the problem, and to Roman Bezrukavnikov for many helpful
conversations. }

\section{Finite W-algebras }

There are a great many references which explain the basic construction
of the finite W-algebras. The papers of Premet \cite{key-21}, Gan-Ginzburg
\cite{key-13}, and Brundan-Goodwin-Kleschev \cite{key-8} all have
very complete introductions. For now we shall just recall the very
basic outline of what we need. 

We let $\mathfrak{g}$ be a complex semisimple lie algebra, and let
$e\in\mathfrak{g}$ be a nonzero nilpotent element. By the Jacobson-Morozov
theorem, there exist $f,h\in\mathfrak{g}$ such that $\{e,f,h\}$
form an $\mathfrak{sl_{2}}$-triple, and we fix such a triple. We
define the Slodowy slice $S\subseteq\mathfrak{g}^{*}$ to be the affine
subspace which corresponds, via the killing isomorphism $\mathfrak{g}\tilde{=}\mathfrak{g}^{*}$,
to the affine space $e+\mbox{ker}(ad(f))$. Slodowy's book \cite{key-24}
contains a wealth of information on these spaces and their uses in
lie theory. Three facts about these spaces are crucial for us. 

The first, recorded in \cite{key-13} section 3, is that this affine
space has a natural Poisson structure inherited from the Poisson structure
on $\mathfrak{g}^{*}$. The second, to be found in \cite{key-13},
section 2, is that the space $S$ admits a natural $\mathbb{C}^{*}$
action defined as follows: our chosen $\mathfrak{sl}_{2}$-triple
gives a homomorphism $\tilde{\gamma}:SL_{2}(\mathbb{C})\to G$, and
we define $\gamma(t)=\tilde{\gamma}\begin{pmatrix}t & 0\\
0 & t^{-1}
\end{pmatrix}$, so that $\mbox{Ad}(\gamma(t))e=t^{2}e$; so we define $\bar{\rho}(t)=t^{-2}\mbox{Ad}^{*}(\gamma(t))$,
a $\mathbb{C}^{*}$-action on $\mathfrak{g}$ which stabilizes $S_{e}$
and fixes $\chi$ (the element of $\mathfrak{g}^{*}$ corresponding
to $e$ under the killing isomorphism). In fact, this action contracts
$S$ to $\chi$. So, we get a grading on $O(S)$ and it is easy to
see that the Poisson multiplication respects this grading. 

Finally, we wish to recall that the space $S$ can be realized as
a {}``Hamiltonian reduction'' of the space $\mathfrak{g}^{*}$.
To explain this, we let $\chi\in\mathfrak{g}^{*}$ be the element
associated to $e$ under the isomorphism $\mathfrak{g}\tilde{=}\mathfrak{g}^{*}$
given by the killing form. We define a skew-symmetric bilinear form
on $\mathfrak{g}(-1)$ via $<x,y>=\chi([x,y])$, which is easily seen
to be nondegenerate. Thus, $(\mathfrak{g}(-1),<,>)$ is a symplectic
vector space, and we choose $l\subset\mathfrak{g}(-1)$ a Lagrangian
subspace. We define $\mathfrak{m}_{l}=l\oplus\bigoplus_{i\leq-2}\mathfrak{g}(i)$,
a nilpotent lie algebra such that $\chi|_{\mathfrak{m}_{l}}$ is a
character of $\mathfrak{m}_{l}$. We let $M_{l}$ be the unipotent
connected algebraic subgroup of $G$ such that $\mbox{Lie}(M_{l})=\mathfrak{m}_{l}$.
We let $I$ denote the ideal of $\mbox{Sym}(\mathfrak{g})=O(\mathfrak{g}^{*})$
generated by $\{m-\chi(m)|m\in\mathfrak{m}_{l}\}$. Then we have an
isomorphism of algebras 
\[
O(S)\tilde{=}(O(\mathfrak{g}^{*})/I)^{M_{l}}
\]
where $M_{l}$ acts via the adjoint action (c.f. \cite{key-13}, lemma
2.1). 

Given this, we can recall that the finite W-algebra associated to
$e\in\mathfrak{g}$, denoted $U(\mathfrak{g},e)$, is a filtered associative
algebra whose associated graded Poisson algebra is isomorphic to $O(S)$.
In fact, the algebra $U(\mathfrak{g},e)$ can be defined as the Hamiltonian
reduction of the enveloping algebra $U(\mathfrak{g})$ in a manner
exactly parallel to the formula above. There is a natural map $Z(U(\mathfrak{g}))\to Z(U(\mathfrak{g},e))$
which is an isomorphism. So we have the usual description of central
characters indexed by elements of the affine space $\mathfrak{h}^{*}/W$
(where $\mathfrak{h}$ is a Cartan subalgebra of $\mathfrak{g}$,
and $W$ is the Weyl group). Given a $\lambda\in\mathfrak{h}$, we
thus get an ideal of $Z(U(\mathfrak{g},e))$ and then an ideal $J_{\lambda}$
of $U(\mathfrak{g},e)$, and we define $U^{\lambda}(\mathfrak{g},e):=U(\mathfrak{g},e)/J_{\lambda}$.

\section{The Cycle Map}

In the paper \cite{key-11}, the author and Kobi Kremnizer gave a
geometric interpretation of certain categories of modules over finite
W-algebras. Inspired by the classical Beilinson-Bernstein localization
theorem, we considered the singular Poisson variety $S\cap\mathcal{N}:=S_{\mathcal{N}}$,
where $\mathcal{N}$ denotes the nilpotent cone of $\mathfrak{g}^{*}$.
The variety $\mathcal{N}$ has a resolution of singularities, denoted
$\mu:\tilde{\mathcal{N}}\to\mathcal{N}$, called the springer resolution
(see \cite{key-9}, chapter 3, for a very complete treatment). It
turns out that the restriction of this map $\mu:\mu^{-1}(S_{\mathcal{N}})\to S_{\mathcal{N}}$
is also a resolution of singularities. We shall denote the scheme
theoretic preimage $\mu^{-1}(S_{\mathcal{N}})$ by $\tilde{S}_{\mathcal{N}}$.
This variety has a natural symplectic structure which extends the
Poisson structure on the base $S_{\mathcal{N}}$. 

Now, the $\mathbb{C}^{*}$ action constructed above lifts naturally
to $\tilde{S}_{\mathcal{N}}$, and it contracts the smooth variety
$\tilde{S}_{\mathcal{N}}$ to the singular variety $\mathcal{B}_{\chi}:=\mu^{-1}(\chi)$,
the springer fibre of $\chi$. This provides the perfect setting to
do geometry. 

In particular, given an anti dominant regular weight $\lambda$, we
constructed a sheaf of $\mathbb{C}[[h]]$- algebras $D_{h}(\lambda,\chi)(0)$
on $\tilde{S}_{\mathcal{N}}$, which is a quantization in the sense
of \cite{key-3}, i.e., it is flat over $\mathbb{C}[[h]]$ and satisfies
$D_{h}(\lambda,\chi)(0)/hD_{h}(\lambda,\chi)(0)\tilde{=}O(\tilde{S}_{\mathcal{N}})$.
Let us describe the relationship between this variety and the finite
$W$-algebra.

The algebra $U^{\lambda}(\mathfrak{g},e)$ is naturally filtered,
as recalled above. Thus we can consider the Rees algebra associated
to this filtered algebra (c.f. \cite{key-4}, section 2.4) which is
naturally an algebra over $\mathbb{C}[h]$. If we then complete with
respect to $h$, we obtain an algebra which we call $U_{h}^{\lambda}(\mathfrak{g},e)(0)$
(One can then formally invert $h$ to obtain a $\mathbb{C}((h))$-algebra
$U_{h}^{\lambda}(\mathfrak{g},e)$, which is one of the main players
in \cite{key-11}, although it won't be used here). We then have 
\begin{equation}
\Gamma(D_{h}(\lambda,\chi)(0))=U_{h}^{\lambda}(\mathfrak{g},e)(0)
\end{equation}
 With these ingredients in hand, we can explain our construction of
the cycle of a finite dimensional module $M$ over $U^{\lambda}(\mathfrak{g},e)$.
Given such, we choose any good filtration $F$ on $M$ (c.f., \cite{key-14},
appendix D). Then we have the module $\mbox{Rees}(M;F)$ and, after
completion, the $\mathbb{C}[[h]]$-module $\widehat{\mbox{Rees}(M;F)}$.
We can then define a localization functor 
\[
\mbox{Loc(}M;F)(0)=D_{h}(\lambda,\chi)(0)\otimes_{U_{h}^{\lambda}(\mathfrak{g},e)(0)}\widehat{\mbox{Rees}(M;F)}
\]
 which makes sense because of 3.1. Because $D_{h}(\lambda,\chi)(0)$
is a quantization of $\tilde{S}_{\mathcal{N}}$, we then get a coherent
sheaf on $\tilde{S}_{\mathcal{N}}$ by letting 
\[
CS(M;F):=\mbox{Loc}(M;F)(0)/h\mbox{Loc}(M;F)(0)
\]

Now we define our main object of study, the cycle map, by 
\[
cc(M):=CC(CS(M;F))\in H_{top}(\mathcal{B}_{\chi},\mathbb{Q})
\]
 where $CC$ stands for the characteristic cycle of a coherent sheaf.
A very complete treatment of characteristic cycles is provided by
the book \cite{key-15}, and some details are recalled in section
8 below. For now, we just recall that there is a chern character map
\[
K(X)\to H_{*}(X)
\]
from the $K$ theory to the total Borel-Moore homology of a projective
scheme $X$. The projection of this map to the top graded piece of
$H_{*}(X)$ yields the map $CC$. 

The fact that, in our case, $CS(M;F)$ is actually supported on $\mathcal{B}_{\chi}$
will also be addressed in section 6 below. 

The fact that the constructon of $cc$ does not depend on the filtration
chosen (while $CS(M;F)$ does) is a standard argument (c.f. \cite{key-14},
appendix D).

\subsection{The Localization Complex is a Sheaf}

For later use, we shall need one important fact about the localization
defined above:
\begin{lem}
The sheaf $\mbox{Loc}(M;F)(0)$ is isomorphic to the complex 
\[
D_{h}(\lambda,\chi)(0)\otimes_{U_{h}^{\lambda}(\mathfrak{g},e)(0)}^{L}\widehat{\mbox{Rees}(M;F)}
\]
In other words, this complex has no cohomology except in degree zero.
\end{lem}
In the paper \cite{key-11}, we showed that this was case after inverting
the parameter $h$. However, this is not enough for our purposes here.
Therefore, we shall obtain the result by using the comparison with
the stronger localization result of \cite{key-32}. To explain this
comparison, let us recall the 
\begin{thm}
(Skryabin's equivalence) Let $Q$ denote the natural $(U(\mathfrak{g}),U(\mathfrak{g},e))$-bimodule
$U(\mathfrak{g})/I$ (c.f. section 2 above). Then the functor 
\[
V\to Q\otimes_{U(\mathfrak{g},e)}V
\]
 is an equivalence of categories from $U(\mathfrak{g},e)$-mod to
the category of $(\mathfrak{m}_{l},\chi)$-equivariant $U(\mathfrak{g},e)$-modules. 
\end{thm}
Let us recall that the latter category is the full subcategory of
$U(\mathfrak{g})$-modules such that the action of the operators $\{m-\chi(m)|m\in\mathfrak{m}_{l}\}$
is locally nilpotent. 

On the other hand, there is a similar relationship between $D_{h}(\lambda,\chi)(0)$-modules
and $D_{h}^{\lambda}(0)$-modules (where $D_{h}^{\lambda}(0)$ is
the sheaf of $h$-completed twisted differential opeators on $\mathcal{B}$;
c.f. section 4 below). Namely, there is a natural pull back functor
$p^{*}:Mod(D_{h}(\lambda,\chi)(0))\to Mod(D_{h}^{\lambda}(0))$ and
we have the 
\begin{thm}
The functor $p^{*}$ is exact and fully faithful. Further, there is
a compatibility 
\[
p^{*}(D_{h}(0,\chi)\otimes_{U_{h}(\mathfrak{g},e)(0)}^{L}\widehat{\mbox{Rees}(M)})\tilde{=}D_{h}(0)\otimes_{U_{h}(0)}^{L}\widehat{\mbox{Rees}(Q\otimes_{U(\mathfrak{g},e)}M)}
\]
 where the filtration on $Q\otimes_{U(\mathfrak{g},e)}M$ is induced
from the filtration on $M$ and the Kazhdan filtration on $Q$. 
\end{thm}
The proof of this theorem follows from the results and constructions
of \cite{key-11}; the same constructions are also explained in \cite{key-17}.
The full image of $p^{*}$ can be described in similar terms to those
of Skryabin's equivalence. 

Therefore, we see that we only have to show that the complex 
\[
D_{h}(0)\otimes_{U_{h}(0)}^{L}\widehat{\mbox{Rees}(Q\otimes_{U(\mathfrak{g},e)}M)}
\]
 is concentrated in a single degree. However, if we equip the algebra
of differential operators $D$ and the universal enveloping algebra
$U$ with the Kazhdan filtration, then the compatibility 
\[
D_{h}(0)\otimes_{U_{h}(0)}^{L}\widehat{\mbox{Rees}(Q\otimes_{U(\mathfrak{g},e)}M)}=\widehat{\mbox{Rees}(D\otimes_{U}^{L}(Q\otimes_{U(\mathfrak{g},e)}M))}
\]

follows by taking a filtered free resolution of $Q\otimes_{U(\mathfrak{g},e)}M$.
Thus our problem is reduced to showing that the complex 
\[
D\otimes_{U}^{L}(Q\otimes_{U(\mathfrak{g},e)}M))
\]
 is concentrated in a single degree. But the functor $M\to D\otimes_{U}M$
is exact on all finitely generated $U(\mathfrak{g})$-modules by Beilinson-Bernstein
localization. So the result follows. 

The main objective of the next few sections will be to relate this
construction of a local object to the positive characteristic machinery
of \cite{key-5}, where the relation between $K$ groups of representations
and homology of springer fibres is very strong indeed.

\section{Localization In positive Characteristic}

In this section, we'll review the main results of the localization
theory for enveloping algebras in characteristic $p$, which can be
found in \cite{key-6} and \cite{key-5}. We recall that the lie algebra
$\mathfrak{g}$ has an integral form $\mathfrak{g}_{\mathbb{Z}}$
(c.f. \cite{key-28}), which then has a base extension to any field
$k$, called $\mathfrak{g}_{k}$. Throughout the rest of the paper,
we will use $k$ to denote an algebraically closed field of positive
characteristic. When $\mbox{char}(k)>h$ (where $h$ is the Coxeter
number of $\mathfrak{g}$), Bezrukavnikov-Mirkovic-Ruminyin have developed
a localization theory for the enveloping algebra $U(\mathfrak{g}_{k})$.
Since this theory is extremely important for us, we shall recall their
basic notations and results in some detail.

\subsection{Quantized Twisted Differential Operators}

We start with the quantized sheaf of twisted differential operators
on $T^{*}\mathfrak{\mathcal{B}}_{\mathbb{C}}$. We first recall that
the original sheaf of twisted differential operators can be defined
using the following two steps (c.f. \cite{key-20}, Chapter C1 for
details): 

First, one defines the sheaf of algebras $U^{0}=O_{\mathfrak{\mathcal{B}}}\otimes_{\mathbb{C}}U(\mathfrak{g})$,
where the multiplication is twisted by the action of an element of
$\mathfrak{g}$, considered as a vector field, on a local section
of $O_{\mathcal{B}}$. Inside $U^{0}$, we have the sub-ideal-sheaf
$\mathfrak{n}^{0}$, which is generated at each point $x\in\mathfrak{\mathcal{B}}$
by the subspace $\mathfrak{n}_{x}\in\mathfrak{g}$ (thinking of $\mathfrak{\mathcal{B}}$
as the variety of Borel subalgebras in $\mathfrak{g}$, each point
gets a Borel $\mathfrak{b}_{x}$ and a corresponding maximal nilpotent
subalgebra; this is $\mathfrak{n}_{x}$). We also, therefore, have
an ideal sheaf $\mathfrak{b}^{0}$ inside $U^{0}$, and containing
$\mathfrak{n}^{0}$. 

From here, we define the sheaf of algebras $D_{\mathfrak{h}}=U^{0}/\mathfrak{n}^{0}$.
Thus there is a natural map from the sheaf of lie algebras $\mathfrak{h}^{0}=\mathfrak{b}^{0}/\mathfrak{n}^{0}$
to the sheaf $D_{\mathfrak{h}}$, which then induces a map $\phi:U(\mathfrak{h})\to\Gamma(\mathfrak{\mathcal{B}},U^{0}/\mathfrak{n}^{0})$.
For any element $\lambda\in\mathfrak{h}^{*}$, we have an ideal $I_{\lambda}\subseteq U(\mathfrak{h})$,
which, due to normalization reasons, is the ideal chosen to correspond
the character $\lambda+\rho$ (where $\rho$ is the half sum of the
positive roots, as usual). Thus we have an ideal $I_{\lambda}D_{\mathfrak{h}}$,
and we can finally put $D^{\lambda}=D_{\mathfrak{h}}/I_{\lambda}D_{\mathfrak{h}}$. 

Now, we can quantize each step of this construction in a natural way.
We start by defining, for an affine cover of $\mathfrak{\mathcal{B}}$,
$\{U_{i}\}$ the sheaves $U_{h}^{0}=O_{U_{i}}\otimes_{\mathbb{C}[h]}U_{h}(\mathfrak{g})$,
where $U_{h}(\mathfrak{g})$ is simply $\mbox{Rees}(U(\mathfrak{g}))$
(with respect to the usual PBW filtration). These are not merely sheaves
on the varieties $U_{i}$, but, by using ore localization, we can
view them as quantizations%
\footnote{In fact, these are not quite quantizations as we have defined them
above, because these algebras are not compete with respect to $h$.
However, they are $h$-free, and their $h$-completions, considered
below, would be quantizations in the strict sense. %
} of the varieties $U_{i}\times\mathfrak{g}$ (we get the structure
algebras of these varieties by setting $h=0$). This construction
glues naturally, and so we get a sheaf $U_{h}^{0}$ on $\mathfrak{\mathcal{B}}\times\mathfrak{g}$. 

Given this, we still have subsheaves $\mathfrak{b}^{0}$ and $\mathfrak{n}^{0}$
generated by the same elements, and so we can consider the quotient
$D_{h,\mathfrak{h}}=U_{h}^{0}/\mathfrak{n}^{0}$, with its associated
map $U_{h}(\mathfrak{h})\to\Gamma(D_{h,\mathfrak{h}})$. Then the
element $\lambda\in\mathfrak{h}^{*}$ still defines an ideal of $U_{h}(\mathfrak{h})$
(defined as the ideal generated by $\{v-h(\lambda+\rho)(v)|v\in\mathfrak{h}\}$)
again called $I_{\lambda}$, and we can now define $D_{h}^{\lambda}=D_{h,\mathfrak{h}}/I_{\lambda}D_{h,\mathfrak{h}}$,
which are now sheaves on the space $T^{*}\mathcal{B}$. We also note
that the sheaf $D_{h,\mathfrak{h}}$ can be considered a sheaf on
the space $\tilde{\mathfrak{g}^{*}}$- the full Grothendeick alteration
(c.f. \cite{key-9}, chapter 3)

\subsection{Differential Operators in Positive Characteristic}

Now suppose, in addition to the assumptions of the above section,
that the element $\lambda\in\mathfrak{h}^{*}$ is integral. Then every
object that we have used in the above construction; the enveloping
algebra $U(\mathfrak{g})$, the group $G$, its Borel and nilpotent
subalgebras, and flag variety, exists over $\mathbb{Z}$. Therefore,
it makes perfect sense to construct the algebra $D_{h}^{\lambda}(\mathbb{Z})$
as a quantization of $T^{*}\mathcal{B}(\mathbb{Z})$, and further,
to base change to an algebraically closed field of positive characteristic,
and thus obtain an object $D_{h}^{\lambda}(k)$ (from now on in this
section we shall drop the $k$, understanding that we are working
over a closed field of positive characteristic) . Upon taking the
quotient $D_{h}^{\lambda}/(h-1)$, we obtain the sheaf of crystalline
differential operators as featured in \cite{key-5}, which we will
simply denote $D^{\lambda}$ . This sheaf has the feature that there
is a {}``frobenius morphism'' 
\[
F:O(T^{*}\mathfrak{\mathcal{B}}^{(1)})\to D^{\lambda}
\]
 obtained by using $p^{th}$ iterates of vector fields, c.f. \cite{key-5},
section 2. 

In the case of the sheaf $D_{h}^{\lambda}$, we can lift $F$ to a
morphism $F:O(T^{*}\mathfrak{\mathcal{B}}^{(1)}\times\mathbb{A}^{1})\to D_{h}^{\lambda}$
(simply by sending the extra variable to $h$). This means that $D_{h}^{\lambda}$
(technically, its $h$-completion) is a {}``Frobenius constant quantization''
in the terminology of \cite{key-3}. Let us recall the definition
there: 
\begin{defn}
Let $O_{h}$ be a quantization of the Poisson scheme $X$ (defined
over $k$). Then $O_{h}$ is a Frobenius constant quantization if
the Frobenius morphism $F:O_{X}^{p}\to O_{X}$ lifts to a morphism
$F:O_{h}\to Z(O_{h})$ (where $Z$ is the algebra center). 
\end{defn}
In this case, the sheaf $O_{h}$ can be regarded as a locally free
coherent sheaf of algebras on the scheme $X^{(1)}\times\mbox{Spec}(k[[h]])$. 

In the case of the sheaf $D_{h,\mathfrak{h}}$, we can even say a
bit more. By the same reasoning, there is a morphism $F:\tilde{\mathfrak{g}^{*}}^{(1)}\to D_{h,\mathfrak{h}}$.
But in fact there is also a morphism in any characteristic (even over
$\mathbb{Z}$) $O(\mathfrak{h}^{*})\to D_{h,\mathfrak{h}}$, simply
by noting that $U(\mathfrak{h})\to D_{h,\mathfrak{h}}$by construction. 

Now, in positive characteristic, both of the schemes $\tilde{\mathfrak{g}^{*}}^{(1)}$
and $\mathfrak{h}^{*}$ live over the scheme $\mathfrak{h}^{*,(1)}$.
The morphism $\tilde{\mathfrak{g}^{*}}^{(1)}\to\mathfrak{h}^{*,(1)}$
is Grothendieck's invariants map (c.f. {[}CG{]} chapter 3), and the
map $\mathfrak{h}^{*}\to\mathfrak{h}^{*,(1)}$ is the Artin-schreier
map for $p-$lie algebras, which on algebras of functions is the morphism
$\mbox{Sym}(\mathfrak{h}^{(1)})\to\mbox{Sym}(\mathfrak{h})$ which
takes $h\to h^{p}-h^{[p]}$ (c.f. \cite{key-5}, section 2.3). 

By comparing images, we see that we actually arrive at a morphism
$\tilde{\mathfrak{g}^{*}}^{(1)}\times_{\mathfrak{h}^{*,(1)}}\mathfrak{h}^{*}\to D_{h,\mathfrak{h}}$,
which will play the role of the Frobenius morphism for this sheaf.

\subsection{Localization For Lie Algebras.}

As in the classical case, localization involves comparing modules
over a sheaf of differential operators to modules over the global
sections. For $\lambda\in\mathfrak{h}_{k}$, we have that 
\[
\Gamma(D^{\lambda})\tilde{=}U^{\lambda}(\mathfrak{g}_{k})
\]

where $U^{\lambda}(\mathfrak{g}_{k})$ is the quotient of the algebra
$U(\mathfrak{g}_{k})$ by the ideal $J_{\lambda}$ obtained as follows:
the algebra $S(\mathfrak{h}_{k})^{(W,\cdot)}$ occurs as a subalgebra
of $Z(U(\mathfrak{g}_{k}))$. As in characteristic zero, this subalgebra
is equal to $U(\mathfrak{g}_{k})^{G_{k}}$, but unlike in characteristic
zero, this is not the whole center. Still, any point in $\mathfrak{h}_{k}$
defines an ideal of $S(\mathfrak{h}_{k})^{(W,\cdot)}$, which can
then be extended to an ideal of $U(\mathfrak{g}_{k})$. 

Then, if $\lambda\in\mathfrak{h}_{k}$ is regular with respect to
the dot-action of $W$, we have (\cite{key-5}): 
\begin{thm}
There is an equivalence of categories 
\[
R\Gamma:D^{b}(mod^{c}(D^{\lambda}))\to D^{b}(mod^{f.g.}(U^{\lambda}))
\]
 where $mod^{c}(D^{\lambda})$ denotes the category of coherent modules
over the Azumaya algebra $D^{\lambda}$ on $T^{*}\mathcal{B}_{k}^{(1)}$.
The inverse functor is given by 
\[
\mathcal{L}^{\lambda}(M)=D^{\lambda}\otimes_{U^{\lambda}}^{L}M
\]

\end{thm}
Let us make a few remarks about the proof of this theorem. The key
ingredients are the following three facts: 

\begin{equation}
R\Gamma(D^{\lambda})\tilde{=}\Gamma(D^{\lambda})\tilde{=}U^{\lambda}
\end{equation}

\begin{equation}
T^{*}\mathcal{B}\mbox{ is variety with trivial canonical class}
\end{equation}

\begin{equation}
U^{\lambda}\mbox{ has finite homological dimension}
\end{equation}
 Given these asumptions, the fact that the two functors are adjoint
is general nonsense. Further, by the definitions, we have $R\Gamma(\mathcal{L}^{\lambda}(U^{\lambda})=R\Gamma(D^{\lambda})=U^{\lambda}$.
This implies, by using free resolutions (and the finite homological
dimension of $U^{\lambda}$), that $\mathcal{L}^{\lambda}$ is a fully
faithful functor. Thus the difficulty is in showing that it is essentially
surjective. This is accomplished by using the assumption on the triviality
of the canonical class. This assumption implies that the Grothendeick-Serre
duality is simply given by the shift functor $[\mbox{dim}\mathcal{B}]$.
This means that the essential image of $\mathcal{L}^{\lambda}$, which
is a triangulated subcategory, is closed under the action of this
functor. This is a very strong categorical condition, which can be
used to show that in fact the essential image of $\mathcal{L}^{\lambda}$
is the entire category.

\subsubsection{Azumaya Splitting}

One major difference between the positive characteristic theory and
the theory in characteristic zero is the fact modules over Azumaya
algebras are closely related to coherent sheaves. To see how this
plays out, one must exploit the full center of the algebra $U(\mathfrak{g}_{k})$.
As mentioned above, the subalgebra $U(\mathfrak{g}_{k})^{G_{k}}$
does not equal the full $Z(U(\mathfrak{g}_{k}))$. The reason for
this is the existence of the so-called $p^{th}$ iterate map $\mathfrak{g}\to\mathfrak{g}^{(1)}$,
denoted $x\to x^{[p]}$ (where $(1)$ again denotes frobenius twist).
In the case of $\mathfrak{gl}(n)$, this map is simply the $p^{th}$
power of a matrix; and one can define it in general (for a reductive
lie algebra) by using a suitable embedding $\mathfrak{g}\to\mathfrak{gl}(n)$. 

This map allows us to construct new central elements as follows: for
any $x\in\mathfrak{g}$, the element $x^{p}-x^{[p]}\in U(\mathfrak{g}_{k})$
is now central. In fact, the algebra generated by these elements,
called the $p$-center of $U(\mathfrak{g}_{k})$, denoted $Z_{p}$,
is isomorphic to $S(\mathfrak{g}^{(1)})$, because the scalar multiplication
on $x^{p}-x^{[p]}$ is twisted by the $p^{th}$ power. Then, one can
describe the full center of the enveloping algebra in positive characteristic
(c.f. \cite{key-5}, chapter 3) as 
\[
Z\tilde{=}O(\mathfrak{g}^{*(1)}\times_{\mathfrak{h}^{*(1)}/W}\mathfrak{h}^{*}/W)
\]
 where $W$ acts via the dot action, and the map from $\mathfrak{h}^{*}\to\mathfrak{h}^{*(1)}$
is the Artin-Schreier map.

Thus, given $\chi\in\mathfrak{g}^{*}$, we can associate to it an
element $\chi^{(1)}\in\mathfrak{g}^{(1),*}$ (c.f. \cite{key-4},
section 3.2), and from this we get a maximal ideal of $S(\mathfrak{g}^{(1)})$,
and a central ideal of $U(\mathfrak{g}_{k})$. Any irreducible module
over $U(\mathfrak{g}_{k})$ will have a well defined central character,
which will restrict to characters of both $Z_{p}$ and $U(\mathfrak{g}_{k})^{G_{k}}$
(the latter is called the Harish-Chandra center of $U(\mathfrak{g}_{k})$).
If we suppose that the Harish-Chandra character is integral, then
this implies that the $Z_{p}$-character is a nilpotent element of
$\mathfrak{g}^{*}$ (\cite{key-5}, section 6). We will assume that
we are in this situation from now on. 

Then the above localization theorem can be modified as follows: Let
$\lambda$ be a regular, integral central character, and let $\chi^{(1)}\in\mathfrak{g}^{*(1)}$
be a nilpotent element. Then we have $\mu^{-1}(\chi)^{(1)}=(\mu^{(1)})^{-1}(\chi^{(1)})$,
the frobenius twist of the springer fibre of $\chi$, which we shall
usually denote $\mathcal{B}_{\chi}$. We define the category $mod_{\chi}^{c}(D^{\lambda})$
to be the category of coherent $D^{\lambda}$ modules which are set-theoretically
supported on the variety $\mathcal{B}_{\chi}^{(1)}$. Similarly, we
define $mod_{\chi}^{f.g.}(U^{\lambda})$ to be the category of $U^{\lambda}$-modules
which have generalized $p$-character $\chi$- i.e., they are killed
by some power of the maximal ideal in $S(\mathfrak{g}_{k})$ associated
to $\chi$. Then we have:
\begin{thm}
. The the above localization theorem specializes to an equivalence
of categories
\[
R\Gamma:D^{b}(mod_{\chi}^{c}(D^{\lambda}))\to D^{b}(mod_{\chi}^{f.g.}(U^{\lambda}))
\]
 
\end{thm}
Finally, we can say a little more about the structure of the category
$mod_{\chi}^{c}(D^{\lambda})$. Since $D^{\lambda}$ is an Azumaya
algebra, it is etale locally a matrix algebra over $O_{T^{*}\mathfrak{\mathcal{B}}^{(1)}}$.
We recall that an Azumaya algebra $\mathfrak{A}$ on a scheme $X$
is said to be \emph{split} if it is in fact isomorphic to a matrix
algebra over $O_{X}$. If this happens, then there is the standard
Morita equivalence of categories $mod^{c}(\mathfrak{A})\tilde{\to}mod^{c}(O_{X})$.
Although $D^{\lambda}$ is not split, we have the following: 
\begin{thm}
Under the same assumptions as above, we have that the restriction
of $D^{\lambda}$ to the formal neighborhood of $\mathcal{B}_{\chi}^{(1)}$
in $T^{*}\mathcal{B}{}^{(1)}$ is a split Azumaya algebra. This implies
an equivalence 
\[
mod_{\chi}^{c}(D^{\lambda})\tilde{\to}mod_{\mathcal{B}_{\chi}^{(1)}}^{c}(O_{T^{*}\mathcal{B}^{(1)}})
\]
And therefore an equivalence 
\[
D^{b}(mod_{\mathcal{B}_{\chi}^{(1)}}^{c}(O_{T^{*}\mathfrak{\mathcal{B}}^{(1)}}))\tilde{\to}D^{b}(mod_{\chi}^{f.g.}(U^{\lambda}))
\]

\end{thm}
Thus we observe a tight relation between representation theory and
coherent sheaves in characteristic $p$. In the next several sections
we relate this theory to the theory of W-algebras.

\section{Localization For Modular W-Algebras}

In this section, we would like to construct the {}``local'' object
which corresponds to the version of the $W$-algebra in positive characteristic
(we shall give a compete definition of this object in section six
below). This object will, in particular, be a sheaf on the scheme
$\tilde{S}_{\mathcal{N}}^{(1)}$ (the Frobenius twist of the resolution
of the Slodowy slice). It will also be necessary to introduce a version
over the extended scheme $\tilde{S}^{(1)}\times_{\mathfrak{h}^{*(1)}}\mathfrak{h}^{*}$,
which will be completely parallel to the first version. We'll begin
by recalling, briefly, the construction of the main object of the
paper \cite{key-11}, which is a quantization of the scheme $\tilde{S}_{\mathcal{N},\mathbb{C}}$.
We shall attempt to keep this paper as self contained as possible.

\subsection{Hamiltonian Reduction}

As indicated above, the W-algebra is related to the lie algebra via
the procedure of Hamiltonian reduction. Since modules over lie algebras
can localized to $D$-modules on the flag variety $\mathcal{B}$,
it stands to reason that modules over W-algebras should localize to
modules over a some sort of Hamiltonian reduction of $D_{\mathfrak{\mathcal{B}}}$.
One version of this (over $\mathbb{C}$) was worked out in \cite{key-11}.
Here, we wish to construct a modular version of this localization,
based on the ideas of \cite{key-5} and the related paper \cite{key-4}. 

In particular, we shall take the Hamiltonian reduction of differential
operators $D_{h}^{\lambda}$ with respect to the action of Premet's
subgroup $M_{\mathfrak{l}}$, and character $\chi$. In particular,
we work under the assumption that $char(k)$ is large enough so that
the lie algebra $\mathfrak{m}_{l}$ (c.f. section 2 above) exists
as over $k$, and satisfies $\mathfrak{m}_{l}^{[p]}=0$ (where $[p]$
is the $p^{th}$ iterate map induced from $\mathfrak{g}$) (note that
this is possible because $\mathfrak{m}_{l}$ is nilpotent). 

We recall that the action of $G$ on the flag variety $\mathfrak{\mathcal{B}}$
induces a {}``quantum'' moment map $\mu:U_{h}(\mathfrak{g})\to D_{h}^{\lambda}$,
which has a restriction, which we shall also call $\mu:U_{h}(\mathfrak{m}_{l})\to D_{h}^{\lambda}$.
In the positive characteristic case, this map also satisfies some
additional structure, as in: 
\begin{defn}
(c.f. \cite{key-3} section 5) 

Suppose that $X$ is a symplectic variety over $k$, and let $O_{h}$
be a Frobenius-constant quantization of $X$%
\footnote{As noted before, our objects defined above actually satisfy the property
that their $h$-completions are quantizations. We shall work with
such completions below.%
}. Suppose further that there is an algebraic group action $H\times X\to X$,
and a morphism $\mu:U_{h}(\mathfrak{h})\to O_{h}$ (where $\mathfrak{h}=\mbox{Lie}(H)$),
satisfying $\mu(h)=h$. 

1) This quantization is said to be {}``Frobenius H-constant'' if
$H$ acts on the sheaf of algebras $O_{h}$ in a way that preserves
the subalgebra $O(X)^{(1)}[[h]]$, such that $H$ fixes the parameter
$h$, and such that the induced action of $H$ on $X^{(1)}$ agrees
with the base change of the action of $H$ on $X$. 

For $\xi\in\mathfrak{h}$, and $s\in O_{h}$ a local section, let
$\xi\cdot s$ denote the action (obtained by differentiating the action
of $H$). 

2) Given the set-up of 1), the map $\mu$ is said to be a quantum
moment map if the restriction of $\mu$ to $U_{h}(\mathfrak{h})^{(1)}=U(\mathfrak{h})^{(1)}[[h]]$
lands in $O(X)^{(1)}[[h]]$, and we have the {}``action'' relation:
for all $\xi\in\mathfrak{h}$, and all local sections $s\in O_{h}$,
$\mu(\xi)s-s\mu(\xi)=h\xi\cdot s$. 
\end{defn}
The last part of definition 2 is the standard definition of a quantum
moment map in any characteristic. The fact that (the completion of)
our map $\mu$ is a quantum moment map follows immediately from the
definition of $\mu$- the algebra $D_{h}^{\lambda}$ was constructed
out of the enveloping algebra. Further, the $p^{th}$ powers on both
sides clearly coincide. 

Given this, we can state the version of Hamiltonian reduction which
we will use, again following \cite{key-3} . In the rest of this section,
we shall work with the $h-$completed algebras $U_{h}(\mathfrak{m}_{l})(0)$
and $D_{h}(\lambda)(0)$ (we use this notation to match our characteristic
zero notation from \cite{key-11}). 

First, we consider the ideal $I_{\chi}\subseteq U_{h}(\mathfrak{m}_{l})(0)$
defined as the ideal generated by $\{m-\chi(m)|m\in\mathfrak{m}_{l}\}$.
Note that this ideal is \emph{not} homogeneous with respect to the
usual grading on $\mbox{Rees}(U_{h})\tilde{=}U_{h}(0)$. 

Next, we define the sheaf of algebras $D_{h}(\lambda)(0)/<\mu(I_{\chi})>$.
This is naturally a sheaf over the scheme $\mu^{-1}(\chi)^{(1)}\times\mbox{Spec}(k[[h]])$.
To see why, note that the central subalgebra $\mbox{Sym}(\mathfrak{m}_{l}^{(1)})[[h]]$
has an ideal generated by the point $\chi^{(1)}\in\mathfrak{m}_{l}^{*}$.
Let $I_{\chi}^{(1)}$ be the ideal in $U_{h}(\mathfrak{m}_{l})(0)$
generated by this central ideal. From the definition of moment map
above, we see that $\mu(I_{\chi}^{(1)})$ defines the subscheme $\mu^{-1}(\chi)^{(1)}$
inside $T^{*}\mathfrak{\mathcal{B}}^{(1)}$. Further, $I_{\chi}^{(1)}=I_{\chi}\cap\mbox{Sym}(\mathfrak{m}_{l}^{(1)})[[h]]$
(recall that we are assuming that $m^{[p]}=0$ for all $m\in\mathfrak{m}_{l}$). 

Finally, we look at the pushforward $p_{*}(D_{h}(\lambda)(0)/<\mu(I_{\chi})>)^{M_{l}}$.
Then this is naturally a sheaf on the scheme $\tilde{S}_{\mathcal{N}}^{(1)}\times\mbox{Spec}(k[[h]])$.
This is our Hamiltonian reduction $D_{h}(\lambda,\chi)(0)$. 

A slight modification gives the case of the sheaf $D_{h,\mathfrak{h}}$:
we note that the moment map works the same way, and the group $M_{l}$
acts on $\tilde{\mathfrak{g}^{*}}^{(1)}$. Then after Hamiltonian
reduction, we end up with a sheaf $D_{h,\mathfrak{h}}(\chi)(0)$ on
the scheme $\tilde{S}^{(1)}\times_{\mathfrak{h}^{*,(1)}}\mathfrak{h}^{*}\times\mbox{Spec}(k[[h]])$
(where $\tilde{S}^{(1)}\to\mathfrak{h}^{*,(1)}$ is the restriction
of the map $\tilde{\mathfrak{g}^{*}}^{(1)}\to\mathfrak{h}^{*,(1)}$). 
\begin{rem}
Below, we will see how to get rid of the parameter $h$ and work with
a sheaf simply defined on the scheme $\tilde{S}_{\mathcal{N}}^{(1)}$.
It would have been possible to do things in the opposite order, i.e.,
get rid of the $h$ on $T^{*}\mathfrak{\mathcal{B}}$ by working with
the crystalline differential operators $D^{\lambda}$, and then take
Hamiltonian reduction of this sheaf to get a sheaf on $\tilde{S}_{\mathcal{N}}^{(1)}$.
This is the approach used for Hilbert schemes in \cite{key-4}, and
it is probably the most natural approach in positive characteristic.
However, this approach doesn't really have an analogue in characteristic
zero, because without including the parameter $h$, the sheaf $D^{\lambda}$
isn't local on $T^{*}\mathfrak{\mathcal{B}}$, and of course we can't
pull back to the frobenius twist because it doesn't exist. Since the
main theorem of this paper involves comparing a construction in characteristic
zero and positive characteristic, we have to include the {}``unnatural''
construction in positive characteristic as well. 
\end{rem}

\subsection{Azumaya Property}

In this section, we use the results of \cite{key-3} (whose set-up
we borrowed in the previous section) to show that the algebra we have
obtained by Hamiltonian reduction is a reasonable object. This will
be the key point in showing that the {}``localization theorem''
holds in our context. 

So, let us quote the results from \cite{key-3} in the form that we
need: 
\begin{prop}
(c.f. \cite{key-3}, proposition 5.8) The sheaf $D_{h}(\lambda,\chi)(0)$
is a Frobenius constant quantization of the variety $\tilde{S}_{\mathcal{N}}$.
Thus it can also be thought of as a coherent sheaf on the scheme $\tilde{S}_{\mathcal{N}}^{(1)}\times\mbox{Spec}(k[[h]])$.
The same is true of $D_{h,\mathfrak{h}}(\chi)(0)$ on $\tilde{S}^{(1)}\times_{\mathfrak{h}^{*,(1)}}\mathfrak{h}^{*}\times\mbox{Spec}(k[[h]])$. 
\end{prop}
and 
\begin{prop}
\label{prop:Azumaya}(c.f. \cite{key-3}, proposition 3.8) Let $O_{h}$
be a frobenius constant quantization of a variety $X$, and let $x$
be a closed point of $X^{(1)}$. Then, regarding $O_{h}(h^{-1})$
as a coherent sheaf of algebras on $X^{(1)}\times\mbox{Spec}(k((h)))$,
the local algebra $O_{h}(h^{-1})_{x}$ is Azumaya over $k((h))$. 
\end{prop}
From these results we derive immediately an Azumaya property for the
sheaf $D_{h}(\lambda,\chi)(0)$. A similar argument shows the Azumaya
property of the sheaf $D_{h,\mathfrak{h}}(\chi)(0)$ (c.f. \cite{key-5}
section 2.3). However, we wish to obtain an even stronger property
by getting rid of the parameter $h$. To do that, we shall invoke
some facts about $\mathbb{G}_{m}$-equivariant quantizations. 

Before doing so, let us note that all of the varieties in the above
section carry the Gan-Ginzburg $\mathbb{G}_{m}$ action on $T^{*}\mathfrak{\mathcal{B}}$,
which we recall is given by 
\[
t(g,v)=(\gamma(t)g,\bar{\rho(t)}v)
\]
 where $\gamma:\mathbb{C}^{*}\to G$ was the natural embedding described
above in section 2, and where we've identified the cotangent space
at the point $g$ with $(\mathfrak{g}/\mathfrak{b}_{0})^{*}\tilde{=}\mathfrak{n}_{0}$,
where $\mathfrak{b}_{0}$ is our standard Borel subalgebra, (which
we choose to contain the {}``positive part'' of our $\mathfrak{sl}_{2}$-triple,
$e$ and $h$), and $\mathfrak{n}_{0}$ is its nilradical, and $\bar{\rho}(t)=t^{-2}ad(\gamma(t))$
as above. We also let $\mathbb{G}_{m}$ act on $\mathfrak{h}^{*}$
as $t\cdot h=t^{2}h$. %
\footnote{This action exists as long as $\mbox{char}k$ is sufficiently large,
c.f \cite{key-16}. We shall assume that $\mbox{char}k$ is large
enough in the rest of the paper.%
} 

We can see also that all of the sheaves considered above are equivariant
with respect to this action: note that the action extends by definition
to $\mathfrak{\mathcal{B}}$, and then to $T^{*}\mathfrak{\mathcal{B}}$
and $D_{h}^{\lambda}$ by the usual extension of an action to differential
operators (as usual, we demand $t\cdot h=t^{2}h$ to make the relations
of $D_{h}^{\lambda}$ homogeneous). In addition, the action preserves
the varieties $S_{\mathcal{N}}$ and $\tilde{S}_{\mathcal{N}}$, and
is respected by the moment map, by its definition. The ideal of Hamiltonian
reduction, which was inhomogeneous with respect to the usual grading,
is homogeneous with respect to this one. Thus we see that $D_{h}(\lambda,\chi)(0)$
carries this action as well, and the same for $D_{h,\mathfrak{h}^{*}}(\chi)(0)$. 

Although this action is poorly behaved on $T^{*}\mathfrak{\mathcal{B}}$,
it is positive weight and contracting on $\tilde{S}_{\mathcal{N}}$.
So at this point we can invoke a very general lemma, which is similar
to \cite{key-3}, lemma 3.4, and \cite{key-18}, proposition 2.1.5: 
\begin{lem}
1) Let $k$ be any algebraically closed field, and let $X$ be smooth
a variety over $k$, with a quantization $O_{h}$. Suppose that $X$
is equipped with a positive weight $\mathbb{G}_{m}(k)$-action, which
extends to an action of $O_{h}$ via $t\cdot h=t^{n}h$ for some $n>0$.
Then the sheaf $O_{h}$ on $X\times\mbox{Spec}(k[[h]])$ is the restriction
of a sheaf on the variety $X\times\mbox{Spec}(k[h])$. 

2) Now let $char(k)>0$ and suppose that $O_{h}$ is frobenius constant,
and, in addition, that $Fr:O(X^{(1)})\to O_{h}$ is $\mathbb{G}_{m}$-equivariant
(where we use the induced action on the Frobenius twist of a variety).
Then we can say in addition that the coherent sheaf $O_{h}$ on $X^{(1)}\times\mbox{Spec}(k[[h]])$
is the restriction of a coherent sheaf on $X^{(1)}\times\mbox{Spec}(k[h])$. \end{lem}
\begin{proof}
To prove 1), we first consider the case that $X$ is affine, following
(\cite{key-18}, prop. 2.1.5). In this case, let $X=\mbox{Spec}(A)$.
As a $k[[h]]$-module, we have that $O_{h}\tilde{=}A[[h]]$. We first
claim that $A[h]=A[[h]]_{\mathbb{G}_{m}-fin}$ (where $(V)_{\mathbb{G}_{m}-fin}$
denotes the sum of the finite dimensional modules of the $\mathbb{G}_{m}$-module
$V$). This simply follows from the obvious fact that every eigenvalue
for the $\mathbb{G}_{m}$-action on $A[[h]]$ is a finite sum of terms
of the form $h^{k}a$ where $a$ is an eigenvalue for the $\mathbb{G}_{m}$-action
on $A$ (here we use that the action is positive weight, so that the
number of $h$'s must be bounded). 

Now, the claim implies, since $\mathbb{G}_{m}$ acts on $A[[h]]$
by algebra automorphisms, that $A[h]$ is a subalgebra of $A[[h]]$.
So this is 1) for affine $X$. In general, we note that taking $\mathbb{G}_{m}$-finite
vectors clearly commutes with localization by a $\mathbb{G}_{m}$-stable
element of $A$, and that any smooth variety with a $\mathbb{G}_{m}$-action
has an affine $\mathbb{G}_{m}$-equivariant cover (c.f., GIT, section
). So we can take the sheaf of local sections of $\mathbb{G}_{m}$-finite
vectors, and this suffices for 1). 

To get 2), we note that the image of $O(X^{(1)})$ clearly lies in
$O_{h,\mathbb{G}_{m}-fin}$ by the assumption. Therefore the extension
to $Fr:O(X^{(1)})[h]\to O_{h}$ obtained by sending $h$ to $h$ has
image in $O_{h,\mathbb{G}_{m}-fin}$ as well. But this is exactly
2). 
\end{proof}
With this in hand, we see right away that in fact $D_{h}(\lambda,\chi)(0)$
is the restriction of a sheaf, called $D_{h}^{\lambda,\chi}$, on
$\tilde{S}_{\mathcal{N}}^{(1)}\times\mathbb{A}^{1}$, and by the same
reasoning, that $D_{h,\mathfrak{h}}(\chi)(0)$ is the restriction
of a sheaf on $\tilde{S}^{(1)}\times_{\mathfrak{h}^{*,(1)}}\mathfrak{h}^{*}\times\mathbb{A}^{1}$.
By \cite{key-3} lemma 3.4, these sheaves are even the unique ones
with this property. Further, we are now free to take the quotient
$D_{h}^{\lambda,\chi}/(h-1)$, (respectively $D_{h,\mathfrak{h}}(\chi)(0)/(h-1)$)
and obtain coherent sheaves on the variety $\tilde{S}_{\mathcal{N}}^{(1)}$
(respectively $\tilde{S}^{(1)}\times_{\mathfrak{h}^{*,(1)}}\mathfrak{h}^{*}$),
which we will call $D^{\lambda,\chi}$ (respectively $\tilde{D}(\chi)$,
following \cite{key-5}). We can now state the main result about these
objects 
\begin{prop}
$D^{\lambda,\chi}$, respectively $\tilde{D}(\chi)$, is an Azumaya
algebra on the variety $\tilde{S}_{\mathcal{N}}^{(1)}$, respectively
$\tilde{S}^{(1)}\times_{\mathfrak{h}^{*,(1)}}\mathfrak{h}^{*}$. \end{prop}
\begin{proof}
This will follow from proposition 9. To see how, let us note (for
the first statement, the second is similar) that it suffices to show
that, for any point $x\in\tilde{S}_{\mathcal{N}}^{(1)}$ with associated
ideal $m_{x}$ in $O(\tilde{S}_{\mathcal{N}}^{(1)})$, $(D^{\lambda,\chi})_{x}/m_{x}(D^{\lambda,\chi})_{x}$
is a central simple algebra (c.f. \cite{key-29} chapter 4; we already
know that these are locally free sheaves because they are frobenius
constant quantizations).

If $\bar{I}$ is a nontrivial ideal of this algebra, then we can lift
it to a nontrivial ideal $I$ of $(D_{h}^{\lambda,\chi})_{x}/m_{x}(D_{h}^{\lambda,\chi})_{x}$.
Since $h-1$ is an element of $\bar{I}$ (by definition of $D^{\lambda,\chi}$),
we see that no power of $h$ can be an element of $\bar{I}$; if it
were, then some $h^{p^{k}}$ would be in $\bar{I}$ ($p=char(k)$),
but since $(h-1)^{p^{k}}=h^{p^{k}}-1$ is in $\bar{I}$, this contradicts
non-triviality. 

Now, this means that the ideal $I$ can be extended to a nontrivial
ideal of 
\[
[D_{h}(\lambda,\chi)(0)(h^{-1})]_{x}/m_{x}
\]
by first completing and then inverting $h$. But now this is a $k((h))$-central
simple algebra by \prettyref{prop:Azumaya}, which is a contradiction. 
\end{proof}

\subsection{Localization Theorem}

With the background of the previous sections, we can now give a proof
of the localization property for the sheaf $D^{\lambda,\chi}$. Throughout
this section, we make the following assumption: the cohomology groups
$H^{i}(\tilde{S}_{\mathcal{N}},O_{\tilde{S}_{\mathcal{N}}})$ and
$H^{i}(\tilde{S},O_{\tilde{S}})$ vanish for $i>0$. This is true
in characteristic zero by the Grauert-Riemenschnieder vanishing theorem,
and hence it is true in sufficiently large positive characteristic.
We do not know an explicit bound. 

First, let us begin with a statement about global sections, to be
proved in section 6 below: 
\begin{prop}
\label{prop:Glob-Secs}We have isomorphisms of algebras: 
\[
\Gamma(\tilde{D}(\chi))\tilde{=}U(\mathfrak{g},e)\otimes_{O(\mathfrak{h}^{*})^{W}}O(\mathfrak{h}^{*})
\]
 and 
\[
\Gamma(D^{\lambda,\chi})\tilde{=}U^{\lambda}(\mathfrak{g},e)
\]
 where these algebras are Premet's modular $W$-algebras, defined
over an algebraically closed field of positive characteristic. The
tensor product in the first line makes sense because $O(\mathfrak{h}^{*})^{W}\tilde{=}Z_{HC}(U(\mathfrak{g}))$
is a central subalgebra of $U(\mathfrak{g},e)$. 
\end{prop}
The reason we need to delay the proof is that these algebras are actually
defined in terms of reduction mod $p$ of the analogous characteristic
zero objects. Thus we have to use the discussion of the reduction
procedure in the next section. Although we do use this proposition
later in this section, the results of the next section are completely
independent of this one. 

In addition, without describing the multiplication on these algebras
explicitly, we can give the following consequences of this definition: 
\begin{lem}
The algebras $\Gamma(\tilde{D}(\chi))$ and \textup{$\Gamma(D^{\lambda,\chi})$}
carry natural filtrations, and we have that $\mbox{gr}\Gamma(\tilde{D}(\chi)))=O(S)\otimes_{S(\mathfrak{h})^{W}}S(\mathfrak{h})$
and $\mbox{gr}(\Gamma(D^{\lambda,\chi}))=O(S_{\mathcal{N}})$. Further,
we have that $\Gamma(\tilde{D}(\chi))\tilde{=}R\Gamma(\tilde{D}(\chi))$
and that $\Gamma(D^{\lambda,\chi})\tilde{=}R\Gamma(D^{\lambda,\chi})$. \end{lem}
\begin{proof}
The fact that the algebras are filtered follows by taking global sections
of the natural filtrations on $\tilde{D}(\chi)$ and $D^{\lambda,\chi}$,
respectively (we recall here that these sheaves of algebras are defined
by taking a quantized algebra mod $h-1$; thus they carry filtrations).
By definition, we have that 
\[
\mbox{gr}(D^{\lambda,\chi})\tilde{=}O(S_{\mathcal{N}})
\]
and 
\[
\mbox{gr}(\tilde{D}(\chi))\tilde{=}O(\tilde{S})
\]

The cohomology vanishing assumption at the beginning of this section
now shows the cohomology vanishing for $D^{\lambda,\chi}$ and $\tilde{D}(\chi)$
by a standard spectral sequence argument. Further, we then see that
the sequences 
\[
0\to\Gamma(\tilde{S}_{\mathcal{N}},D_{i}^{\lambda,\chi})\to\Gamma(\tilde{S}_{\mathcal{N}},D_{i+1}^{\lambda,\chi})\to\Gamma(\tilde{S}_{\mathcal{N}},O(\tilde{S}_{\mathcal{N}})_{i+1})\to0
\]
 (and the analogous one for $\tilde{S}$) are exact for all $i$.
This shows the statements about the associated graded algebras. 
\end{proof}
Now we are almost in a position to state and prove the localization
theorems which are relevant to this paper. In particular, we shall
show that the alebras in question satisfy the assumptions 4.1, 4.2,
and 4.3 above. Given that, the argument of \cite{key-5} and \cite{key-3}
applies verbatim. 

The cohomology vanishing (assumption 4.1) is discussed above. 

The triviality of the canonical class (assumption 4.2) comes from
the fact that we are working with algebraic symplectic varieties. 

Thus, the only remaining obstacle is the issue of the algebras having
finite homological dimension. We shall get around this in the same
way as \cite{key-5}, section 3, whose notation and proofs we shall
follow very closely in the sequel. 

First we define the localization functor 
\[
\mathcal{L}:D^{b}(mod^{f.g.}(U(\mathfrak{g},e))\to D^{b}(mod^{coh}(\tilde{D}(\chi))
\]
 as $\mathcal{L}(M)=\tilde{D}(\chi)\otimes_{U(\mathfrak{g},e)}^{L}M$.
We note that the above proposition makes $U(\mathfrak{g},e)$ a subalgebra
of $\tilde{W}(\chi)$, and further that $U(\mathfrak{g},e)$ is a
filtered algebra whose associated graded is isomorphic to $O(S)$
(the coordinate ring of affine space). Thus this algebra has finite
homological dimension. Therefore this definition makes sense.

Next, we note that there is an action of $O(\mathfrak{h}^{*})$ on
the sheaf $\mathcal{L}(M)$ (via its action on $\tilde{D}(\chi)$),
while there is only an action of $O(\mathfrak{h}^{*})^{W}$ on $M$.
So, for $\lambda\in\mathfrak{h}^{*}$, we can define the category
$mod_{\lambda}^{f.g.}(U(\mathfrak{g},e))$ of modules such that the
algebra $O(\mathfrak{h}^{*})^{W}$ acts by a generalized central character
$\lambda$ (i.e., the image of $\lambda$ in $\mathfrak{h}^{*}/W$),
and there is a decomposition 
\[
\mathcal{L}(M)\tilde{=}\bigoplus_{\mu\in W\cdot\lambda}\mathcal{L}^{\mu\to\lambda}(M)
\]
 via the action of the generalized characters in $O(\mathfrak{h}^{*})$.
We wish to study the functor $\mathcal{L}^{\lambda\to\lambda}(M)$
when $\lambda$ is a regular element of $\mathfrak{h}^{*}$. 

If we define the category $mod_{\lambda}^{coh}(\tilde{D}(\chi))$
to be the full subcategory of $mod^{coh}(\tilde{D}(\chi))$ on objects
such that $O(\mathfrak{h}^{*})$ acts with generalized central character
$\lambda$, then we note that the image of $\mathcal{L}^{\lambda\to\lambda}$
lands in $D^{b}(mod_{\lambda}^{coh}(\tilde{D}(\chi)))$. We shall
denote this functor by 
\[
\mathcal{L}^{\hat{\lambda}}:D^{b}(mod_{\lambda}^{f.g.}(U(\mathfrak{g},e)))\to D^{b}(mod_{\lambda}^{coh}(\tilde{D}(\chi)))
\]
We note right off the bat that this functor takes makes sense on bounded
derived categories since it is defined as a summand of a functor which
does the same. 

Our other important functor will be the functor 
\[
\mathcal{L}^{\lambda}:D^{-}(mod^{f.g.}(U(\mathfrak{g},e)^{\lambda})\to D^{-}(mod^{coh}(D^{\lambda,\chi}))
\]
 defined as $\mathcal{L}^{\lambda}(M)=D^{\lambda,\chi}\otimes_{U(\mathfrak{g},e)^{\lambda}}^{L}M$.
Our aim is to show that in fact this functor has finite homological
dimension. This will be accomplished once we prove 
\begin{lem}
Suppose $\lambda$ is regular. Then we have a compatibility between
$\mathcal{L}^{\lambda}$ and $\mathcal{L}^{\hat{\lambda}}$; in other
words, if we consider the inclusions $i:D^{-}(mod^{f.g.}(U(\mathfrak{g},e)^{\lambda})\to D^{-}(mod_{\lambda}^{f.g.}(U(\mathfrak{g},e))$
and $\iota:D^{-}(mod^{coh}(D^{\lambda,\chi}))\to D^{-}(mod_{\lambda}^{coh}(\tilde{D}(\chi))$
then we have 
\[
\iota\mathcal{L}^{\lambda}\tilde{=}\mathcal{L}^{\hat{\lambda}}i
\]
 
\end{lem}
Obviously, this lemma proves the needed claim that $\mathcal{L}^{\lambda}$
preserves the bounded derived categories. 
\begin{proof}
(of the lemma). The key point is to rewrite the functor $\mathcal{L}^{\hat{\lambda}}$
in a way that makes it closer to $\mathcal{L}^{\lambda}$. To do that,
we first define the sheaf 
\[
\tilde{D}(\chi)^{\hat{\lambda}}:=\tilde{D}(\chi)\otimes_{O(\mathfrak{h}^{*})}O(\mathfrak{h}^{*})^{\hat{\lambda}}
\]
 where $O(\mathfrak{h}^{*})^{\hat{\lambda}}$ is the completion of
the ring $O(\mathfrak{h}^{*})$ at the ideal generated by $\lambda$.
Then for any $M\in mod_{\lambda}^{f.g.}(U(\mathfrak{g},e))$, the
definitions yield 
\begin{equation}
\mathcal{L}^{\hat{\lambda}}(M)\tilde{=}\tilde{D}(\chi)^{\hat{\lambda}}\otimes_{U(\mathfrak{g},e)}M
\end{equation}
 On the other hand, we have by definition 
\[
D^{\lambda,\chi}=\tilde{D}(\chi)\otimes_{O(\mathfrak{h}^{*})}k_{\lambda}
\]
 (where $k_{\lambda}$ is the one dimensional $O(\mathfrak{h}^{*})$-module
corresponding to the maximal ideal $\lambda$). Now, since $\lambda$
is regular, the projection $\mathfrak{h}^{*}\to\mathfrak{h}^{*}/W$
is etale at $\lambda$, and so there is an isomorphism 
\[
O(\mathfrak{h}^{*})^{\hat{\lambda}}\otimes_{O(\mathfrak{h}^{*}/W)}k_{\lambda}\tilde{=}k
\]
 and so we deduce 
\[
\tilde{D}(\chi)^{\hat{\lambda}}\otimes_{U(\mathfrak{g},e)}^{L}U(\mathfrak{g},e)^{\lambda}\tilde{=(}\tilde{D}(\chi)\otimes_{O(\mathfrak{h}^{*})}O(\mathfrak{h}^{*})^{\hat{\lambda}})\otimes_{U(\mathfrak{g},e)}^{L}U(\mathfrak{g},e)^{\lambda}\tilde{=}D^{\lambda,\chi}
\]
 because of the isomorphism $U(\mathfrak{g},e)^{\lambda}\tilde{=}U(\mathfrak{g},e)\otimes_{O(\mathfrak{h}^{*}/W)}k_{\lambda}$.
But this equivalence is precisely the isomorphism of functors that
we wanted, after writing out the definitions of the localization functors,
and using the realization 6.1. 
\end{proof}
With this taken care of, we now have: 
\begin{thm}
The functors 
\[
R\Gamma:D^{b}(mod_{\lambda}^{coh}(\tilde{D}(\chi)))\to D^{b}(mod_{\lambda}^{f.g.}(U(\mathfrak{g},e)))
\]
 and 
\[
R\Gamma:D^{b}(mod^{coh}(D^{\lambda,\chi}))\to D^{b}(mod^{f.g.}(U(\mathfrak{g},e)^{\lambda}))
\]
 are equivalences of categories, with the inverse functors given by
$\mathcal{L}^{\hat{\lambda}}$ and $\mathcal{L}^{\lambda}$, respectively. 
\end{thm}

\subsection{Restriction to a Springer Fibre}

Analogously to \cite{key-5}, chapter 4, (c.f. section 4 above) we
can consider the above equivalence of categories after restriction
to the springer fibre $\mathcal{B}_{\chi}^{(1)}$. In particular,
we can define categories $mod_{\mathcal{B}_{\chi}^{(1)}}^{coh}(\tilde{D}(\chi))$
and $mod_{\mathcal{B}_{\chi}^{(1)}}^{coh}(D^{\lambda,\chi})$ for
sheaves which are set theoretically supported on the variety $\mathcal{B}_{\chi}^{(1)}$.
On the representation-theoretic side, we should restrict to those
representations on which the central subalgebra $O(\tilde{S}^{(1)})$
acts via the generalized character $\chi^{(1)}$. We denote this category
by $mod_{\chi}(U(\mathfrak{g},e))$. Then from the above equivalences
of categories we immediately deduce the following 
\begin{thm}
We have the following equivalences of categories: 

\[
R\Gamma:D^{b}(mod_{\lambda,\mathcal{B}_{\chi}^{(1)}}^{coh}(\tilde{D}(\chi)))\to D^{b}(mod_{\lambda,\chi}^{f.g.}(U(\mathfrak{g},e)))
\]

\[
R\Gamma:D^{b}(mod_{\mathcal{B}_{\chi}^{(1)}}^{coh}(D^{\lambda,\chi}))\to D^{b}(mod_{\chi}^{f.g.}(U^{\lambda}(\mathfrak{g},e)))
\]

\end{thm}

\subsection{Azumaya Splitting}

Our aim in this section is to give a brief explanation of the structure
of our Azumaya algebras upon restriction to the springer fibre $\mathcal{B}_{\chi}^{(1)}$.
The result, which will follow from the analogous one in \cite{key-5},
is the following: 
\begin{thm}
a) For all $\lambda\in\mathfrak{h}^{*}$, the Azumaya algebra $\tilde{D}(\chi)$
splits on the formal neighborhood of $\mathcal{B}_{\chi}^{(1)}\times_{\mathfrak{h}^{*(1)}}\lambda$
in $\tilde{S}^{(1)}\times_{\mathfrak{h}^{*(1)}}\mathfrak{h}^{*}$. 

b) Let $M_{\chi}^{\lambda}$ be the vector bundle appearing in \cite{key-5},
theorem 5.1.1, and $E_{\chi}^{\lambda}$ be the vector bundle appearing
in part a). Let $i$ denote the inclusion map $\tilde{S}^{(1)}\times_{\mathfrak{h}^{*(1)}}\mathfrak{h}^{*}\to\tilde{\mathfrak{g}}\times_{\mathfrak{h}^{*(1)}}\mathfrak{h}^{*}$.
Then there is a vector space $V$, of rank $p^{dim\mathcal{B}-dim\mathcal{B}_{e}}$,
such that $i^{*}M_{\chi}^{\lambda}\tilde{=}E_{\chi}^{\lambda}\otimes_{k}V$. 
\end{thm}
The proof of this theorem relies on an examination of the proof of
theorem 5.1.1 in \cite{key-5}. In particular, the argument there
relies on an analysis of the generic structure of the Azumaya algebra
$\tilde{D}$, which works as follows: 

We let $\mathfrak{h}_{unr}^{*}$ denote the open subset of $\mathfrak{h}^{*}$
consisting of those $\lambda$ such that for any coroot $\alpha$
we have either $<\alpha,\lambda+\rho>=0$ or $<\alpha,\lambda>\notin\mathbb{F}_{p}$.
These are called the unramified weights. Then Brown and Gordon \cite{key-7}
described the structure of $U(\mathfrak{g})$ as an algebra over the
scheme $\mathfrak{Z}_{unr}:=\mathfrak{g}^{*(1)}\times_{\mathfrak{h}^{*(1)}/W}\mathfrak{h}_{unr}^{*}$
as follows: 
\begin{prop}
The algebra $U(\mathfrak{g})\otimes_{\mathfrak{Z}}\mathfrak{Z}_{unr}$
is Azumaya over $\mathfrak{Z}_{unr}$. 
\end{prop}
In \cite{key-5}, chapter 3, they deduce the following (which is recorded
there as proposition 5.2.1b)) 
\begin{prop}
\label{prop:azumaya}$U(\mathfrak{g})\otimes_{\mathfrak{Z}}O(\tilde{\mathfrak{g}^{*}}^{(1)}\times_{\mathfrak{h}^{*(1)}}\mathfrak{h}_{unr}^{*})\tilde{\to}\tilde{D}|_{\tilde{\mathfrak{g}^{*}}^{(1)}\times_{\mathfrak{h}^{*(1)}}\mathfrak{h}_{unr}^{*}}$. 
\end{prop}
Thus, if the weight $\lambda$ is unramified, the restriction of $\tilde{D}$
to the formal neighborhood $\mathcal{B}_{\chi}^{(1)}\times_{\mathfrak{h}^{*(1)}}\lambda$,
is the pullback of an Azumaya algebra on the formal neighborhood of
the point $\chi^{(1)}$ in $\mathfrak{g}^{*(1)}$. Since every Azumaya
algebra over the formal neighborhood of a point is split, this implies
the existence of $M_{\chi}^{\lambda}$ in this case. One can deduce
the general case from this one, by noting that there is a functor
of {}``twist by a character'' which interchanges different weights
(\cite{key-5}, lemma 2.3.1). 

To see how to deduce the result in our case, we apply Hamiltonian
reduction to both sides of proposition \ref{prop:azumaya}. We immediately
arrive at the isomorphism 
\[
U(\mathfrak{g},e)\otimes_{S^{*(1)}\times_{\mathfrak{h}^{*(1)}/W}\mathfrak{h}^{*}}O(\tilde{S{}^{*}}^{(1)}\times_{\mathfrak{h}^{*(1)}}\mathfrak{h}_{unr}^{*})\tilde{\to}\tilde{D}(\chi)|_{\tilde{S{}^{*}}^{(1)}\times_{\mathfrak{h}^{*(1)}}\mathfrak{h}_{unr}^{*}}
\]

Thus we will be able to finish the argument the same way if we can
show that the algebra appearing on the left hand side is Azumaya,
at least upon restriction to the formal neighborhood of the point
$\chi^{(1)}$. This is indeed the case, and we can argue as follows:
let $U(\mathfrak{g},e)_{\chi}$ denote the quotient of the algebra
$U(\mathfrak{g},e)$ by the ideal generated by point $\chi^{(1)}\in\tilde{S}^{(1)}$,
and let $U(\mathfrak{g})_{\chi}$ denote the enveloping algebra at
the $p$-character $\chi$. Then Premet in \cite{key-21} has given
an isomorphism 
\[
U(\mathfrak{g})_{\chi}\tilde{=}\mbox{Mat}_{p^{d(e)}}(k)\otimes_{k}U(\mathfrak{g},e)_{\chi}
\]
 where $d(e)$ is the number $\mbox{dim}(\mathcal{B})-\mbox{dim}(\mathcal{B}_{\chi})$.
Thus the restriction $U(\mathfrak{g},e)_{\chi}^{\lambda}$ (with $\lambda$
unramified) is indeed a matrix algebra over $k$, as required. The
general theorem (part a) now follows by the same {}``twist by a character''
argument as in \cite{key-5}. Part b) follows immediately from the
above isomorphism.

Finally, combining the last two sections we arrive at the following
equivalences: 
\begin{thm}
There are equivalences of categories, for regular integral $\lambda$: 

\[
D^{b}(Coh_{\mathcal{B}_{\chi}^{(1)}\times\lambda}(\tilde{S^{(1)}}\times_{\mathfrak{h}^{*(1)}}\mathfrak{h}^{*}))\tilde{\to}D^{b}(mod_{\lambda,\mathcal{B}_{\chi}^{(1)}}^{coh}(\tilde{D}(\chi)))\tilde{\to}D^{b}(mod_{\lambda,\chi}^{f.g.}(U(\mathfrak{g},e)))
\]

\[
D^{b}(Coh_{\mathcal{B}_{\chi}^{(1)}}(\tilde{S}_{\mathcal{N}}))\tilde{\to}D^{b}(mod_{\mathcal{B}_{\chi}^{(1)}}^{coh}(D^{\lambda,\chi}))\tilde{\to}D^{b}(mod_{\chi}^{f.g.}(U^{\lambda}(\mathfrak{g},e)))
\]

\end{thm}

\section{Reduction Mod P}

Since our problem involves relating certain constructions over $\mathbb{C}$
with those over fields of positive characteristic, we shall have to
give a construction of all our objects over a ring $A$ which is finitely
generated over $\mathbb{Z}$. We shall have to follow very closely
Premet's work \cite{key-22} on the modular $W$-algebras, since part
of our goals involve relating our constructions to his. 

We begin by recalling the sheaf $D_{h}(\lambda)(0)_{\mathbb{Z}}$,
which is a quantization of the scheme $T^{*}\mathcal{B}(\mathbb{Z})$.
By base extension, this yields a sheaf $D_{h}(\lambda)(0)_{A}$ for
an arbitrary ring $A$, which is a quantization of $T^{*}\mathcal{B}(A)$. 

At this point, we would like to ape the construction given in the
previous section, and define a quantization of $\tilde{S}_{\mathcal{N}}(A)$
via Hamiltonian reduction of the sheaf $D_{h}(\lambda)(0)_{A}$. This
will work, as long as we make a suitable choice of the ring $A$. 

As a first approximation to what we want, recall from the previous
section that all of the objects needed to define the finite $W$-algebra-
the enveloping algebra, the nilpotent subalgebra $m_{l}$, its associated
group $M_{l}$, are defined and free over a ring $A=\mathbb{Z}[S^{-1}]$
where $S$ is a finite set of primes which depends on the type of
$\mathfrak{g}$ and the choice of nilpotent element. Then we can define
an ideal sheaf of $D_{h}(\lambda)(0)_{A}$ 
\[
I=D_{h}(\lambda)(0)_{A}\cdot<m-\chi(m)|m\in\mathfrak{m}_{l,A}>
\]
 where an element $m\in m_{l,A}$ acts on $D_{h}(\lambda)(0)_{A}$
through the natural action of $\mathfrak{g}_{A}$ (c.f. the definition
of $D_{h}(\lambda)(0)_{A}$ given above). From here, the naive Hamiltonian
reduction can be defined as: 
\[
D_{h}(\lambda,\chi)(0)_{A}:=p_{*}\mathfrak{E}nd{}_{D_{h}(\lambda)(0)_{A}}(D_{h}(\lambda)(0)_{A}/I)
\]
 where $p:\mu^{-1}(\chi)\to\tilde{S}_{\mathcal{N}}$ is the morphism
of Hamiltonian reduction, as above. When $A=\mathbb{C}$ then this
is just the usual reduction of differential operators, defined, e.g.,
in \cite{key-17} and used in \cite{key-11}. Unfortunately, for an
arbitrary ring $A$, this object might not be flat over the variety
$\tilde{S}_{e,A}$, and hence is unsatisfactory for use in arguments
where we have to reduce mod $p$. 

To get around this problem, we shall use the technique of defining
an $A$-lattice of the sheaf $D_{h}(\lambda,\chi)_{\mathbb{C}}$ (for
suitable $A$) which will be flat by construction. For reasonable
rings $A$, this construction will agree with the naive one introduced
above. 

We start with the {}``global'' case, which is Premet's construction.

\subsection{Premet's Construction}

In this subsection, we give a slight variant of the construction of
Premet. We start with the finite $W$-algebra over $\mathbb{C}$,
$U(\mathfrak{g},e)$. We know that $U(\mathfrak{g},e)$ is a filtered
algebra, and it satisfies $\mbox{gr}(U(\mathfrak{g},e))\tilde{=}O(S)$.
Let's consider the Rees algebra of $U(\mathfrak{g},e)$ with respect
to this filtration, which we shall denote $U_{h}(\mathfrak{g},e)$
(this is uncompleted with respect to $h$). Then we have that $U_{h}(\mathfrak{g},e)/h\tilde{=}O(S)$. 

So, we choose homogeneous elements $\{X_{i}\}$ which generate $U_{h}(\mathfrak{g},e)$
as a $\mathbb{C}[h]$-module and whose images mod $h$ form a homogeneous
algebraically independent generating set for $O(S)$. The relations
for the elements $X_{i}$ involve finitely many complex numbers. Therefore,
we can choose a ring $A$, finitely generated over $\mathbb{Z}$,
which contains all constants for these relations. Thus we define a
ring $U_{h,A}(\mathfrak{g},e)$, which is a finitely generated graded
$A[h]$-algebra and which satisfies $U_{h,A}(\mathfrak{g},e)\otimes_{A}\mathbb{C}\tilde{=}U_{h}(\mathfrak{g},e)$.
Then this ring is clearly a free $A[h]$-module by construction, since
$U_{h}(\mathfrak{g},e)$ is a free $\mathbb{C}[h]$-module. 

By possibly making a finite extension of $A$, we can also demand
something more. Recall that there is an injection of algebras $\mathbb{C}[\mathfrak{h}^{*}]^{W}\to U_{h}(\mathfrak{g},e)$
via the action of the center $Z(U(\mathfrak{g}))$ and the Harish-Chandra
isomorphism. We choose the ring $A$ so that there is an injection
$A[\mathfrak{h}^{*}]^{W}\to U_{h,A}(\mathfrak{g},e)$, which, when
base-changed to $\mathbb{C}$, becomes the action of the center. 

Next, we define the algebra $U_{h}(\mathfrak{g},e)_{A}$ as the {}``naive''
Hamiltonian reduction $(U_{h}(\mathfrak{g}_{A})/I)^{M_{l,A}}$ as
in the previous section. Then we have the standard realization 
\[
U_{h}(\mathfrak{g},e)_{A}\tilde{=}End_{U_{h}(\mathfrak{g}_{A})}(Q_{h,\chi}(A))
\]
 where $Q_{h,\chi}(A)=U_{h}(\mathfrak{g}_{A})/I$ as a left $U_{h}(\mathfrak{g}_{A})$-module.
Then $Q_{h,\chi}(A)$ is an $A$-lattice in $Q_{h,\chi}(\mathbb{C})$.
So, enlarging $A$ if necessary, we may assume that each of the $X_{i}$
preserve $Q_{h,\chi}(A)$. This yields a map $U_{h,A}(\mathfrak{g},e)\to U_{h}(\mathfrak{g},e)_{A}$. 

Then, just as in \cite{key-22}, page 7, we have the
\begin{claim}
The map $U_{h,A}(\mathfrak{g},e)\to U_{h}(\mathfrak{g},e)_{A}$ is
an isomorphism. 
\end{claim}
This shows, incidentally, that the construction of this algebra was
independant of the various choices (at least if $A$ is large enough).
From this it easily follows that for any base change to a field $A\to k$,
we have $U_{h,,A}(\mathfrak{g},e)\otimes_{A}k\tilde{=}U_{h}(\mathfrak{g},e)_{k}$. 

Now we shall generalize all this to the local case. 

To do this, we momentarily work over $\mathbb{C}$ again: we have
by section 6.4 above that the quantization $D_{h}(\lambda,\chi)(0)_{\mathbb{C}}$
was the restriction of a sheaf of algebras $D_{h}^{\lambda,\chi}(\mathbb{C})$
on the scheme $\tilde{S}_{\mathcal{N}}\times\mathbb{A}_{\mathbb{C}}^{1}$,
which of course satisfies 
\[
D_{h}^{\lambda,\chi}(\mathbb{C})/hD_{h}^{\lambda,\chi}(\mathbb{C})\tilde{=}O_{\tilde{S}_{\mathcal{N}}}
\]
It is well-known that the scheme $\tilde{S}_{\mathcal{N}}$ is defined
and flat over $A$ for suitable $A$ (in particular $\mathbb{Z}[S^{-1}]$
where $S$ is a finite set of primes will suffice) . For a given finite
open affine cover of $\tilde{S}_{\mathcal{N},A}$, denoted $U_{i}$,
we can choose, consistently with the cover, generators of $O(\tilde{S}_{\mathcal{N},A},U_{i})$
(as $A$-algebras), and then choose (consistently with the cover)
lifts of these to $D_{h}^{\lambda,\chi}(\mathbb{C})(U_{i})$. 

Now we proceed exactly as above: we regard these elements as living
inside $p_{*}\mathfrak{E}nd_{D_{h}^{\lambda}}(D_{h}^{\lambda}/I)$,
we extend $A$ so that they preserve the subspace $(D_{h}^{\lambda}(A)/I_{A})$.
Because $D_{h}^{\lambda}(A)/I_{A}$ is finitely generated over $D_{h}^{\lambda}(A)$,
the resulting $A$ can be chosen to be finitely generated over $\mathbb{Z}$.
Then, we can look at the $A[h]$ algebra generated by these elements
inside each $D_{h}^{\lambda,\chi}(\mathbb{C})(U_{i})$. Because of
the consistency conditions specified above, these glue together to
form a sheaf on $\tilde{S}_{\mathcal{N},A}$ which will be denoted
$D_{h,A}^{\lambda,\chi}$. It is clearly a subsheaf of $D_{h}^{\lambda,\chi}(A)$.
In addition, since $D_{h}^{\lambda,\chi}(\mathbb{C})$ is a free finitely
generated $\mathbb{C}[h]$-algebra, we clearly have that $D_{h,A}^{\lambda,\chi}$
is free and finitely generated over $A[h]$. By the construction we
see that 
\[
D_{h,A}^{\lambda,\chi}/hD_{h,A}^{\lambda,\chi}\tilde{=}O(\tilde{S}_{\mathcal{N},A})
\]
and 
\[
D_{h,A}^{\lambda,\chi}\otimes_{A}\mathbb{C}\tilde{=}D_{h}^{\lambda,\chi}(\mathbb{C})
\]

In a completely parallel fashion, we can define the sheaf $\tilde{D}_{h,A}(\chi)$
as a sheaf of algebras on the variety $\tilde{S}_{A}$. We choose
$A$ so that there is an injection of algebras $A[\mathfrak{h}^{*}]\to\Gamma(\tilde{D}_{h,A}(\chi))$,
which base changes to the corresponding map over $\mathbb{C}$. 

So, we can argue exactly as above to show 
\begin{claim}
Let $k$ be any algebraically closed field of positive characteristic
such that there is a morphism $A\to k$. Then we have an isomorphism
\[
D_{h,A}^{\lambda,\chi}\otimes_{A}k\tilde{=}D_{h}^{\lambda,\chi}(k)
\]
 there is also an isomorphism 
\[
\tilde{D}_{h,A}(\chi)\otimes_{A}k\tilde{=}\tilde{D}_{h,k}(\chi)
\]

\end{claim}

\subsection{Global Sections}

In this section, we shall prove the unproved assertions of the previous
chapter. Above we have defined algebras $U_{h,A}(\mathfrak{g},e)$. 
\begin{defn}
The algebra $U(\mathfrak{g}_{A},e)$ is defined to be $U_{h,A}(\mathfrak{g},e)/(h-1)$.
By the discussion in the previous section, this agrees with the definition
given in \cite{key-22}. 
\end{defn}
We wish to compare these algebras to the global sections of the {}``local''
versions that we also introduced. This is a straightforward matter
given the discussion above- we simply choose $A$ large enough so
that there are isomorphisms $\Gamma(D_{h}^{\lambda}(A)/I_{A})\tilde{=}U^{\lambda}(\mathfrak{g}_{A})/I_{A}$
and $\Gamma(\tilde{D}_{h}(A)/I_{A})\tilde{=}U(\mathfrak{g}_{A})\otimes_{A[\mathfrak{h}^{*}]^{W}}A[\mathfrak{h}^{*}]/I_{A}$.
Then, since everything is defined as operators on these spaces, we
immediately deduce isomorphisms of naive Hamiltonian reductions 
\[
\Gamma(\tilde{D}_{h}(\chi)(A))\tilde{=}U_{h}(\mathfrak{g},e)_{A}\otimes_{A[\mathfrak{h}^{*}]^{W}}A[\mathfrak{h}^{*}]
\]
and 
\[
\Gamma(D_{h}^{\lambda,\chi}(A))\tilde{=}U_{h}(\mathfrak{g},e)_{A}\otimes_{A[\mathfrak{h}^{*}]^{W}}A_{\lambda}:=U_{h}^{\lambda}(\mathfrak{g},e)_{A}
\]
where by $A_{\lambda}$ we mean the $A[\mathfrak{h}^{*}]^{W}$- module
corresponding to the maximal ideal generated by the integral weight
$\lambda$. By choosing appropriate bases, we then get isomorphisms
of our flat quantizations 
\[
\Gamma(\tilde{D}_{h,A}(\chi))\tilde{=}U_{h,A}(\mathfrak{g},e)\otimes_{A[\mathfrak{h}^{*}]^{W}}A[\mathfrak{h}^{*}]
\]
 and 
\[
\Gamma(D_{h,A}^{\lambda,\chi})\tilde{=}U_{h,A}(\mathfrak{g},e)\otimes_{A[\mathfrak{h}^{*}]^{W}}A_{\lambda}:=U_{h,A}^{\lambda}(\mathfrak{g},e)
\]

By base-change to $k$, we see that these isomorphisms yield immediately
\prettyref{prop:Glob-Secs} in the previous section.

\subsection{Localization over $A$ }

In this subsection, we shall describe a localization functor which
lives over the ring $A$. In particular, the isomorphisms of the previous
section allow us to define 
\[
\mathcal{L}_{A}^{\lambda}:D^{b}(mod^{f.g.}(U_{h,A}(\mathfrak{g},e)^{\lambda}))\to D^{b}(mod^{coh}(D_{h,A}^{\lambda,\chi}))
\]
 via 
\[
\mathcal{L}_{A}^{\lambda}(M)=M\otimes_{U_{h,A}(\mathfrak{g},e)^{\lambda}}^{L}D_{h,A}^{\lambda,\chi}
\]
\footnote{It might not be immediately obvious that this functor lands in the
bounded derived category. But we can argue as in the previous section
(reducing to the whole ring $U_{h,A}(\mathfrak{g},e)$) to see that
it does.%
}For the application we have in mind, we will start with an $M\in Mod^{f.d.}(U^{\lambda}(\mathfrak{g},e)_{\mathbb{C}})$.
We choose a good filtration $F$ on $M$, and using it we arrive at
a module $\mbox{Rees}(M_{A})$ over the ring $U_{h,A}^{\lambda}(\mathfrak{g},e)$,
and hence a complex of sheaves, its localization, which by abuse of
notation we denote $\mathcal{L}_{A}^{\lambda}(M_{A})$. 

Now, from the definitions we have a compatibility 
\[
\widehat{\mathcal{L}_{A}^{\lambda}(M_{A})\otimes_{A}\mathbb{C}}\tilde{=}\mbox{Loc}(M;F)
\]
 our original localization sheaf. From the result of section 3.1,
we deduce that the complex $\mathcal{L}_{A}^{\lambda}(M_{A})\otimes_{A}\mathbb{C}$
is actually concentrated in degree zero (i.e., is a sheaf). Since,
locally on $\tilde{S}_{\mathcal{N}}$, $\mathcal{L}_{A}^{\lambda}(M_{A})$
is a finite module over an algebra which is finitely generated over
$A$, we deduce that after making another (finite) exension of $A$,
the complex $\mathcal{L}_{A}^{\lambda}(M_{A})$ is in fact a sheaf
as well, which is flat as a module over $A$ and $h$. 

From here, we can define the sheaf $CS(M_{A}):=\mathcal{L}_{A}^{\lambda}(M_{A})/h$,
which is a coherent sheaf on $\tilde{S}_{\mathcal{N},A}$. In particular,
we have a class $[CS(M_{A})]\in K(\tilde{S}_{A})$, whose base change
to $\mathbb{C}$ is the class $[CS(M)]$. 

We can also take the base change to any algebraically closed field
$k$ (of positive characteristic); taking the quotient of this by
$h-1$ then yields the localization functor $\mathcal{L}_{k}^{\lambda}(M_{k})$
discussed above (we note that taking the quotient by $h-1$ is actually
an equivalence of categories in positive characteristic by \cite{key-3},
lemma 3.4). From this we deduce immediately the following compatibility:
$CS(\mathcal{L}^{\lambda}(M_{k}))=\mathcal{L}_{A}^{\lambda}(M_{A})\otimes_{A}k/(h)$
where $CS$ denotes the sheaf we get after taking associated graded
with respect to the induced filtration (as a $D^{\lambda,\chi}$-module).

We wish to analyze the support of the sheaf $\mathcal{L}_{A}^{\lambda}(M_{A})\otimes_{A}k$.
To this end, we have the 
\begin{prop}
The support of $(\mathcal{L}_{A}^{\lambda}(M_{A})\otimes_{A}k)/h$,
as a closed subset of $\tilde{S}_{k}$, is simply the image under
$Fr$ of the support of $\mathcal{L}^{\lambda}(M_{k})$ in $\tilde{S}_{k}^{(1)}$.
In fact, we can even say that $Fr_{*}[CS(\mathcal{L}^{\lambda}(M_{k})]=[\mathcal{L}^{\lambda}(M_{k})]$
in $K(\tilde{S}_{k}^{(1)})$.%
\footnote{Here we are simply regarding the sheaf $\mathcal{L}^{\lambda}(M_{k})$
as a coherent sheaf on $\tilde{S}^{(1)}$in the naive way; we are
not invoking any Azumaya splitting. %
}\end{prop}
\begin{proof}
This is essentially just a restatement of the fact that $D^{\lambda,\chi}$
is a Frobenius constant quantization. For this tells us that $F:O(\tilde{S}_{k}^{(1)})\to D^{\lambda,\chi}$
becomes the Frobenius morphism after taking the associated graded.
On the other hand, pulling back under $F$ is exactly how we arrive
at the support of $\mathcal{L}^{\lambda}(M_{k})$ as a sheaf on $\tilde{S}_{k}^{(1)}$.
The refined result follows from the rational invariance of $K$-theory.
We are comparing the classes of two coherent sheaves on $\tilde{S}^{(1)}$.
Both are obtained from the $h$-flat sheaf $\mbox{Rees}(\mathcal{L}_{k}^{\lambda}(M_{k}))$
on $\tilde{S}^{(1)}\times\mathbb{A}^{1}$, the first by restriction
to $h=0$, the second by restriction to $h=1$. 
\end{proof}
Now, since $M$ is a finite dimensional module over $U^{\lambda}(\mathfrak{g},e)$,
its support over $S$ is simply the point $\chi$. After reduction
mod $p$, for $p$ sufficiently large, this implies that $M_{k}$
is a module in $mod_{\chi}(U^{\lambda}(\mathfrak{g},e))$- simply
because the map $O(S_{k}^{(1)})\to U(\mathfrak{g},e)_{k}$ becomes
the frobenius morphism after taking gr (as in the proposition). 

So, combining this discussion with the theorems in the previous section,
we see that $\mathcal{L}_{A}^{\lambda}(M_{A})/(h)$ is supported,
set theoretically, on $\mathcal{B}_{\chi}$- and hence the same is
true of the base change $\mathcal{L}_{\mathbb{C}}^{\lambda}(M)$.
So now we can state definitively our {}``base change'' lemma: 
\begin{lem}
The class $[CS(M_{A})]$ actually lives in $K(Coh_{\mathcal{B}_{\chi,A}}(\tilde{S}_{A})=K(\mathcal{B}_{\chi,A})$.
Its pullback to $K(\mathcal{B}_{\chi,k})$ induces the class $[CS(\mathcal{L}^{\lambda}(M_{k}))]$,
and the pullback to $\mathbb{C}$ induces $[CS(M)]$. 
\end{lem}
Since the {}``specialization morphism'' $K(\mathcal{B}_{\chi,\mathbb{C}})\to K(\mathcal{B}_{\chi,k})$
(c.f. \cite{key-5}, chapter 7) is an isomorphism (and the same is
true of the induced map $H_{*}(\mathcal{B}_{\chi,k})\to H_{*}(\mathcal{B}_{\chi,\mathbb{C}})$),
and the {}``reduction mod p'' map $K(mod^{f.d.}(U^{\lambda}(\mathfrak{g},e))\to K(mod_{\chi}^{f.g.}(U^{\lambda}(\mathfrak{g},e))$
is injective, see that we have reduced the injectivity problem to
the following 
\begin{thm}
The restriction of the morphism $K(mod_{\chi}^{f.g.}(U^{\lambda}(\mathfrak{g},e))\to K(\mathcal{B}_{\chi,k})\to H_{top}(\mathcal{B}_{\chi,k})$
to the image of $K(mod^{f.d.}(U^{\lambda}(\mathfrak{g},e))\to K(mod_{\chi}^{f.g.}(U^{\lambda}(\mathfrak{g},e))$
is injective. \end{thm}
\begin{rem}
We should comment here that for all $p$ sufficiently large, there
are isomorphisms between the groups $K(\mathcal{B}_{\chi,k})$ obtained
via the comparison with characteristic zero. If we choose a ring $A$
which works simultaneously for all simple finite dimensional $U^{\lambda}(\mathfrak{g},e)$-modules,
then we see that we can choose any algebraically closed field of large
positive characteristic to show injectivity. We make such a choice
from now on.
\end{rem}
The proof of this theorem will occupy the next section.

\section{K-theory}

We shall need to recall a few facts from the algebraic $K$-theory
developed in \cite{key-15} and \cite{key-1}. In particular, recall
that if we have a proper scheme $Y$, and a closed embedding $Y\to X$,
where $X$ is smooth, then we have a {}``localized chern character''
map 
\[
ch_{Y}^{X}:K_{\mathbb{Q}}(Y)\to A_{\mathbb{Q}}(Y)
\]
 where $A_{\mathbb{Q}}(Y)$ denotes the rational chow ring of algebraic
cycles. The map is obtained by realizing $A(Y)$ as a ring of cycles
on $X$ which are supported on $Y$. This map has the following functorial
property: if $Z\to W$ is another inclusion of a proper into a smooth
variety, and $f:X\to W$ restricts to a morphism from $Y$ to $Z$,
then we have 
\[
f^{*}ch_{Z}^{W}=ch_{Y}^{X}f^{*}
\]
 There is also the {}``Riemann-Roch'' morphism $\tau:K_{\mathbb{Q}}(Y)\to A_{\mathbb{Q}}(Y)$;
this morphism respects proper pushforward. The two morphisms are of
course quite different, but in both cases the projection to the {}``top''
piece $A_{dim(Y)}(Y)$ yields the same map: the algebraic cycle map. 

As $A(Y)$ is a graded vector space, we shall denote by $(ch_{X}^{Y})_{i}$
and $\tau_{i}$ the maps obtained after projection to the degree $i$
cycles. 

All of the varieties we will consider below (namely, the springer
fibres) have the property that their cohomology is spanned by the
classes of algebraic cycles (c.f. the main results of \cite{key-10}).
Thus we have a degree doubling isomorphism 
\[
A_{\mathbb{Q}_{l}}(Y)\to H_{*}(Y)
\]
 where we must now consider the etale Borel-Moore homology group.
Given this, we shall mainly work with the groups $A_{\mathbb{Q}}(Y)$
from now on. From this condition on our varieties it also follows
that the morphisms $ch_{Y}^{X}$ and $\tau$ are isomorphisms. 

With all this in hand, we can proceed to the proof of theorem 29.
First of all, we have a functor $mod_{\chi}^{coh}(D^{\lambda,\chi})\to mod_{\chi}^{coh}(D^{\lambda})$,
given by $\mathcal{F}\to V\otimes\mathcal{F}$ (c.f. section $6.7$
for the vector space $V$ of dimension $p^{d(e)}$). 

Next, we recall the Azumaya splitting of section six. Let us denote
the coherent sheaf obtained from $M_{k}$ via this splitting $Coh(M_{k})$.
So we have that 
\[
E_{\chi}^{\lambda}\otimes Coh(M_{k})\tilde{=}\mathcal{L}^{\lambda}(M_{k})
\]
 and therefore (by section 5.5) that 
\[
M_{\chi}^{\lambda}\otimes Coh(M_{k})\tilde{=}V\otimes\mathcal{L}^{\lambda}(M_{k})
\]
 where $M_{\chi}^{\lambda}$ is the splitting bundle of \cite{key-5}
(note here that $Coh(M_{k})$ is scheme theoretically supported on
$\tilde{S}_{\mathcal{N}}$). But the class in $K$-theory of this
bundle has already been studied. If we make the normalization following
\cite{key-27}, then by section 6 of \cite{key-5}, we in fact have
\[
[M_{\chi}^{\lambda}]=[((Fr_{\mathcal{B}})_{*}O_{\mathcal{B}})|_{\mathcal{B}_{\chi}^{(1)}}]
\]
where we are now considering $\mathcal{B}_{\chi}$ as a subvariety
of $\mathcal{B}$ (i.e., the class on the right lives in $K_{\mathcal{B}_{\chi}^{(1)}}(\mathcal{B}^{(1)})$
which is isomorphic to $K(\mathcal{B}_{\chi}^{(1)})$). Combining
these equalities with push-pull now yields the following: 
\begin{lem}
We have the following equalities in $K(\mathcal{B}_{\chi}^{(1)})$:
\[
[\mathcal{L}^{\lambda}(M_{k})]=p^{-d(e)}[((Fr_{\mathcal{B}})_{*}O_{\mathcal{B}})|_{\mathcal{B}_{\chi}^{(1)}}][Coh(M_{k})]=p^{-d(e)}(Fr_{\mathcal{B}})_{*}(Fr_{\mathcal{B}})^{*}[Coh(M_{k})]
\]
 where we use the fact that the action of $Fr$ on $\mathcal{B}$
takes $\mathcal{B}_{\chi}$ to $\mathcal{B}_{\chi}^{(1)}$. 
\end{lem}
Now, we would like to combine this information with our equality from
the previous section 
\[
Fr_{*}[CS(\mathcal{L}^{\lambda}(M_{k})]=[\mathcal{L}^{\lambda}(M_{k})]
\]
 In this equality however, we were considering the frobenius with
respect to scheme $\tilde{S}_{\mathcal{N}}^{(1)}$. However, the discrepancy
is rectified by the following 
\begin{claim}
$(Fr_{\tilde{S}_{\mathcal{N}}})_{*}[CS(\mathcal{L}^{\lambda}(M_{k})]$
and $(Fr_{\mathcal{B}})_{*}[CS(\mathcal{L}^{\lambda}(M_{k})]$ agree
as classes in $K(\mathcal{B}_{\chi}^{(1)})$. \end{claim}
\begin{proof}
As $CS(\mathcal{L}^{\lambda}(M_{k}))$ is set-theoretically supported
on $\mathcal{B}_{\chi}$, it has a finite filtration by sheaves which
are scheme-theoretically supported there. This filtration means that
we can write the class $[CS(\mathcal{L}^{\lambda}(M_{k})]$ as a sum
of classes $[A_{i}]$ of sheaves on the scheme $\mathcal{B}_{\chi}$. 

However, the proof for sheaves of this type is simply an examination
of the definition of the frobenius morphism, which makes it clear
that the restriction of $Fr_{\tilde{S}_{\mathcal{N}}}$ to $\mathcal{B}_{\chi}$
and $Fr_{\mathcal{B}}$ to $\mathcal{B}_{\chi}$ coincide- in fact
they are both $Fr_{\mathcal{B}_{\chi}}$. 
\end{proof}
So now we are free to compare the previous inequalities and deduce
that 
\[
(Fr_{\mathcal{B}})_{*}[CS(\mathcal{L}^{\lambda}(M_{k})]=p^{-d(e)}(Fr_{\mathcal{B}})_{*}(Fr_{\mathcal{B}})^{*}[Coh(M_{k})]
\]
 in $K(\mathcal{B}_{\chi}^{(1)})$. 

To prove the theorem, it remains to study the action of the operators
$(Fr_{\mathcal{B}})_{*}$ and $(Fr_{\mathcal{B}})^{*}$ on $K$-theory.
To start, we have 
\begin{prop}
The map $(Fr_{\mathcal{B}})_{*}:K(\mathcal{B}_{\chi})\to K(\mathcal{B}_{\chi}^{(1)})$
is injective. \end{prop}
\begin{proof}
The map $Fr$ is finite, flat, and bijective on closed points. After
application of the map $\tau$, which commutes with proper pushforward,
we see that it is enough to check our claim at the level of algebraic
cycles. But it is immediate from the aforementioned properties of
the map $Fr$ that it takes a set of linearly independent cycles to
another. Further, we know that the chow groups of Springer fibres
are spanned by such cycles.
\end{proof}
Combining the proposition with the previous equality, we arrive an
equality in $K(\mathcal{B}_{\chi})$: 
\[
[CS(\mathcal{L}^{\lambda}(M_{k})]=p^{-d(e)}(Fr_{\mathcal{B}})^{*}[Coh(M_{k})]
\]
 Now, we must evaluate the operator $(Fr_{\mathcal{B}})^{*}$. To
this end, we shall translate the problem to $A(\mathcal{B}_{\chi})$
via the localized chern character $ch_{\mathcal{B}_{\chi}}^{\mathcal{B}}$.
We see that we have equalities 
\[
(ch_{\mathcal{B}_{\chi}}^{\mathcal{B}})_{i}(Fr_{\mathcal{B}})^{*}=(Fr_{\mathcal{B}})^{*}(ch_{\mathcal{B}_{\chi}^{(1)}}^{\mathcal{B}^{(1)}})_{i}
\]
in $A_{i}(\mathcal{B}_{\chi})$. 

Further, we can make an identification $A(\mathcal{B}_{\chi})\tilde{=}A(\mathcal{B}_{\chi}^{(1)})$
by the isomorphism of abstract varieties $\mathcal{B}_{\chi}\tilde{=}\mathcal{B}_{\chi}^{(1)}$.
Then we have the
\begin{prop}
The map $(Fr_{\mathcal{B}})^{*}|_{A_{i}(\mathcal{B}_{\chi}^{(1)})}$
followed by by the identification $A_{i}(\mathcal{B}_{\chi})\tilde{=}A_{i}(\mathcal{B}_{\chi}^{(1)})$
is simply multiplication by $p^{dim\mathcal{B}-i}$. \end{prop}
\begin{proof}
As above, since $Fr$ is finite, flat, and bijective on points, it
is clear that the pull-back takes a cycle to a multiple of itself.
Thus, it suffices to check the multiplicity locally, on an open subset
of a cycle (which is equivalent to checking at the generic point).
So, we let $V^{(1)}$ be an algebraic subvariety of $\mathcal{B}_{\chi}^{(1)}$
of dimension $i$, and $x\in V^{(1),sm}$ its smooth locus. We consider
the completed local ring at such a point, $\widehat{\mathcal{O}}_{x,V^{(1)}}\tilde{=}k[[x_{1}^{p},...,x_{i}^{p}]]$.
Then we can compute the pullback under $Fr$ as:
\[
\hat{\mathcal{O}}_{x,\mathcal{B}}\otimes_{\hat{\mathcal{O}}_{x,\mathcal{B}^{(1)}}}\widehat{\mathcal{O}}_{x,V^{(1)}}\tilde{=}k[[x_{1},...x_{dim(B)}]]/(x_{i+1}^{p},...,x_{dim(B)}^{p})
\]
 which is $\mathcal{\hat{O}}_{x,V}$-module of rank $p^{dim(B)-i}$.
The identification $A_{i}(\mathcal{B}_{\chi})\tilde{=}A_{i}(\mathcal{B}_{\chi}^{(1)})$
simply sends the cycle $V$ to $V^{(1)}$. This proves the proposition.
\end{proof}
Now, we let $[Coh(M_{k})]^{(-1)}$ be the class in $K(\mathcal{B}_{\chi})$
which corresponds to $[Coh(M_{k})]$ under the identification $K(\mathcal{B}_{\chi})\tilde{=}K(\mathcal{B}_{\chi}^{(1)})$
(we choose the identification of these groups which is compatible
under $ch_{\mathcal{B_{\chi}}}^{\mathcal{B}}$ and $ch_{\mathcal{B}_{\chi}^{(1)}}^{\mathcal{B}^{(1)}}$
with the identification $A(\mathcal{B}_{\chi})\tilde{=}A(\mathcal{B}_{\chi}^{(1)})$
used above). Then the proposition combined with our previous equality
yields 
\[
(ch_{\mathcal{B_{\chi}}}^{\mathcal{B}})_{i}[CS(\mathcal{L}^{\lambda}(M_{k})]=p^{-d(e)}p^{dim(\mathcal{B})-i}(ch_{\mathcal{B_{\chi}}}^{\mathcal{B}})_{i}[Coh(M_{k})]^{(-1)}=p^{dim(\mathcal{B}_{\chi})-i}(ch_{\mathcal{B_{\chi}}}^{\mathcal{B}})_{i}[Coh(M_{k})]^{(-1)}
\]
 (the last equality is simply because $d(e)=\mbox{dim}(\mathcal{B})-\mbox{dim}(\mathcal{B}_{\chi})$). 

Now we finish the proof of the theorem: for $\mbox{char}(k)$ sufficiently
large, the class $[CS(\mathcal{L}^{\lambda}(M_{k})]$ is independent
of $p$ (it is defined as the reduction of a class over $A$). Further,
the main result of \cite{key-6}, chapter 5, asserts that the same
is true of the class $[Coh(M_{k})]^{(-1)}$. So we deduce immediately
the equalities 
\[
(ch_{\mathcal{B_{\chi}}}^{\mathcal{B}})_{i}[CS(\mathcal{L}^{\lambda}(M_{k})]=0
\]
 for $i\neq\mbox{dim}(\mathcal{B}_{\chi})$ and 
\[
(ch_{\mathcal{B_{\chi}}}^{\mathcal{B}})_{dim(\mathcal{B}_{\chi})}[CS(\mathcal{L}^{\lambda}(M_{k})]=(ch_{\mathcal{B_{\chi}}}^{\mathcal{B}})_{dim(\mathcal{B}_{\chi})}[Coh(M_{k})]^{(-1)}
\]
 So from this we conclude that 
\[
(ch_{\mathcal{B_{\chi}}}^{\mathcal{B}})_{dim(\mathcal{B}_{\chi})}[CS(\mathcal{L}^{\lambda}(M_{k})]=(ch_{\mathcal{B_{\chi}}}^{\mathcal{B}})_{dim(\mathcal{B}_{\chi})}[Coh(M_{k})]^{(-1)}=(ch_{\mathcal{B_{\chi}}}^{\mathcal{B}})[Coh(M_{k})]^{(-1)}
\]

The map on the left is exactly the characteristic cycle map. In addition,
the map on the right is the image in $K$-theory of the equivalence
of categories 
\[
D^{b}(mod_{\chi}^{f.g.}(U(\mathfrak{g},e)^{\lambda})\to D^{b}(Coh_{\mathcal{B}_{\chi}}(\tilde{S}_{\mathcal{N}}))
\]
 and thus is obviously injective. This proves Theorem ?, and thus
the injectivity in characteristic zero.

\section{$W$-Equivariance}

This section is devoted to a discussion of the equivariance of the
characteristic cycle with respect to the action of the Weyl group
$W$. We begin with the following 
\begin{thm}
\cite{key-23} For any algebraically closed field $k$ of sufficiently
large positive characteristic (or characteristic zero), the variety
$\tilde{S}_{\mathcal{N},k}$ admits a (weak) action of the braid group
$B$ of type $\mathfrak{g}$. This action induces an action of the
Weyl group $W$ on $K_{\mathbb{Q}}(\mathcal{B}_{\chi})$, where it
is equivalent to Springer's representation. 
\end{thm}
Thus, combining this with the result of the previous section, we see
that it suffices to show that the $W$ action on $K(mod^{f.g.}(U^{\lambda}(\mathfrak{g},e)))$,
upon reduction mod $p$, agrees with the above defined action on the
category of coherent sheaves. 

Now, it is already known (c.f. \cite{key-23}) that the braid group
action on coherent sheaves agrees, under the localization equivalence
of \cite{key-5} with the one coming from the translation functors
on $U(\mathfrak{g})$-modules. Therefore, what remains to be done
is to compare the translation functor action on $mod^{f.g.}(U(\mathfrak{g},e))$
with the one on $mod^{f.g.}(U(\mathfrak{g}))$, and to look at the
compatibility with the localization functors of this paper. We start
by giving a brief review of the theory of translation functors for
$U(\mathfrak{g},e)$-modules, as developed in \cite{key-35}.

\subsection{Translation Functors for W-algebras}

Let $V$ be a finite dimensional complex representation of $\mathfrak{g}$.
Then, as is well known, the functor $M\to M\otimes_{\mathbb{C}}V$
is a well defined endofunctor of the category $mod^{f.g.}(U(\mathfrak{g}))$,
by putting the usual tensor product action of $\mathfrak{g}$ on $M\otimes_{\mathbb{C}}V$.
After decomposing this functor with respect to the action of the center
of $U(\mathfrak{g})$, one obtains the translation functors $mod^{f.g.}(U^{\hat{\lambda}}(\mathfrak{g}))\to mod^{f.g.}(U^{\hat{\mu}}(\mathfrak{g}))$
for central characters $\lambda$, $\mu$. As is well known, the same
theory exists in positive characteristic, under the assumption that
$V$ integrates to a $G$-module for the associated simply connected
algebraic group $G$ (this assumption is of course vacuous in characteristic
zero). 

In the paper \cite{key-35}, Goodwin has defined a compatible theory
of translation functors for the finite $W$-algebras. Namely, he defines
translation on $U(\mathfrak{g},e)$-modules via the Skryabin equivalence
and translation on $U(\mathfrak{g})$-modules. Explicitly, for $M\in mod^{f.g.}(U(\mathfrak{g},e))$
and $V$ a finite dimensional integrable $\mathfrak{g}$-module, define
\[
M\star V:=\mbox{Wh}_{\chi}((Q\otimes_{U(\mathfrak{g},e)}M)\otimes_{k}V)
\]
 where $Q$ is the bimodule featured in section 3 above, and $\mbox{Wh}_{\chi}$
is the functor which associates to an $(\mathfrak{m}_{\mathfrak{l}},\chi)$-integrable
$\mathfrak{g}$-module the space of vectors such that $\mathfrak{m}_{\mathfrak{l}}$
acts with strict character $\chi$. We recall that this functor is
the inverse to Skryabin's equivalence. It is explained in \cite{key-35}
that this procedure is well defined (i.e., that the adjoint action
of $\mathfrak{m}_{\mathfrak{l}}$ on $(Q\otimes_{U(\mathfrak{g},e)}M)\otimes_{k}V$
is $(\mathfrak{m}_{\mathfrak{l}},\chi)$ integrable) and really gives
an endofunctor of $mod^{f,g.}(U(\mathfrak{g},e))$. We have been vague
about the ground field; the procedure works the same way over any
algebraically closed field for which the algebra $U(\mathfrak{g},e)$
and the bimodule $Q$ are defined (c.f. section 6 above). Furthermore,
since the Skryabin equivalence takes $U(\mathfrak{g},e)$ modules
of a given generalized central character to $U(\mathfrak{g})$ modules
with the same generalized central character, the translation functors
for finite $W$-algebras give rise to translation functors on categories
of $U^{\hat{\lambda}}(\mathfrak{g},e)$ modules. 

Now let us specialize to the case where $\mbox{char}(k)>0$. In this
case, we wish to work with the finer decompositions of our categories
given by the $p$-character. Namely, we wish to show 
\begin{lem}
Suppose that $M$ is a $U(\mathfrak{g},e)$ module with generalized
$p$-character $\chi$. Then $M\star V$ also has generalized $p$-character
$\chi$. \end{lem}
\begin{proof}
By definition, we have $Q=U(\mathfrak{g})/I_{\chi}$. Therefore, this
is a finite module over the algebra $S(\mathfrak{g}^{(1)})/I_{\chi}^{(1)}=\mbox{Spec}(\mu^{-1}(\chi)^{(1)})$
via the Frobenius embedding $S(\mathfrak{g}^{(1)})\to U(\mathfrak{g})$.
Let $p^{(1)}:\mu^{-1}(\chi)^{(1)}\to S^{(1)}$ denote the natural
projection (recall that this is an $M_{\mathfrak{l}}^{(1)}$-bundle).
Then the (set-theoretic) $p$-support of the $U(\mathfrak{g})$-module
$Q\otimes_{U(\mathfrak{g},e)}M$ will be the variety $(p^{(1)})^{-1}(\mbox{Supp}(M))$,
where $\mbox{Supp}(M)$ denotes the (set theoretic) $p$-support of
$M$. 

On the other hand, $V$ is an integrable module, and so has trivial
$p$-support. Therefore $(Q\otimes_{U(\mathfrak{g},e)}M)\otimes_{k}V)$
has $p$-support equal to $(p^{(1)})^{-1}(\mbox{Supp}(M))$ as well. 

Now, if $M$ has $p$-support exactly $\chi^{(1)}$, then $(Q\otimes_{U(\mathfrak{g},e)}M)\otimes_{k}V)$
has $p$-support $(p^{(1)})^{-1}(\chi^{(1)})$. Since this fibre is
a single copy of the Frobenius twisted group scheme $M_{\mathfrak{l}}^{(1)}$,
if we then take the Whittaker invariants functor $\mbox{Wh}_{\chi}$,
the resulting module will again have $p$-support $\chi^{(1)}$. 
\end{proof}
Now, this result tells us that, for a given finite $G$-module $V$,
we have well defined maps $T_{V}:K(Mod^{f.g.}(U^{\hat{\lambda}}(\mathfrak{g},e)_{\hat{\chi}})\to K(Mod^{f.g.}(U^{\hat{\mu}}(\mathfrak{g},e)_{\hat{\chi}})$
for any integral central characters $\lambda$ and $\mu$, which agree
with the translation functors in characteristic zero for any module
obtained via reduction mod $p$. On the other hand, the theory of
translation functors for $U(\mathfrak{g})$ modules gives us maps
$T_{V}:K(Mod^{f.g.}(U^{\hat{\lambda}}(\mathfrak{g})_{\hat{\chi}})\to K(Mod^{f.g.}(U^{\hat{\mu}}(\mathfrak{g})_{\hat{\chi}})$.
Further, Premet's equivalence of categories recalled in section 5.5
above gives us isomorphisms $K(Mod^{f.g.}(U^{\hat{\lambda}}(\mathfrak{g},e)_{\hat{\chi}})\tilde{\to}K(Mod^{f.g.}(U^{\hat{\lambda}}(\mathfrak{g})_{\hat{\chi}})$.
Therefore, we wish to assert the obvious compatbility. This isn't
quite the same as the previous lemma, since the module $Q\otimes_{U(\mathfrak{g},e)}M$
does not belong to $Mod^{f.g.}(U^{\hat{\lambda}}(\mathfrak{g})_{\hat{\chi}}$. 

To get what we want, let us recall that the equivalence of categories
\[
Mod^{f.g.}(U^{\hat{\lambda}}(\mathfrak{g},e)_{\chi})\tilde{\to}Mod^{f.g.}(U^{\hat{\lambda}}(\mathfrak{g})_{\chi})
\]
 is given by the functor $M\to Q_{\chi}\otimes_{U(\mathfrak{g},e)_{\chi}}M$,
where $Q_{\chi}$ is the $(U(\mathfrak{g})_{\chi},U(\mathfrak{g},e)_{\chi})$-bimodule
given by $k_{\chi^{(1)}}\otimes_{S(\mathfrak{g}^{(1)})/I_{\chi}^{(1)}}Q$
; i.e., the reduction of $Q$ to $p$-character $\chi$. Therefore
Premet's functor can be expressed as the functor $M\to k_{\chi^{(1)}}\otimes_{S(\mathfrak{g}^{(1)})/I_{\chi}^{(1)}}Q\otimes_{U(\mathfrak{g},e)}M$.
The inverse is the functor $\mbox{Wh}_{\chi}$, which applies as usual
to $Mod^{f.g.}(U^{\hat{\lambda}}(\mathfrak{g})_{\chi})$. Thus the
compatibility we seek comes down to the 
\begin{claim}
For a finite dimensional $G$-module $V$, we have isomorphisms of
$U(\mathfrak{g},e)$-modules
\[
\mbox{Wh}_{\chi}(k_{\chi^{(1)}}\otimes_{S(\mathfrak{g}^{(1)})/I_{\chi}^{(1)}}Q\otimes_{U(\mathfrak{g},e)}M\otimes_{k}V)\tilde{\to}\mbox{Wh}_{\chi}(Q\otimes_{U(\mathfrak{g},e)}M\otimes_{k}V)
\]

\end{claim}
the proof of which consists of unravelling the definitions (any module
obtained as $\mbox{Wh}_{\chi}(Q\otimes_{U(\mathfrak{g},e)}M\otimes_{k}V)$
already satisfies the requisite $p$-character condition, as in the
proof of the above lemma, so the application of $k_{\chi^{(1)}}\otimes_{S(\mathfrak{g}^{(1)})/I_{\chi}^{(1)}}$
is superfluous).

\subsection{Compatibility}

As recalled above, the braid group action to the category $ $$D^{b}(Coh_{\mathcal{B}_{\chi}^{(1)}}(\tilde{S}_{\mathcal{N}}))$
is defined via restriction from the one on $D^{b}(Coh_{\mathcal{B}_{\chi}^{(1)}}(\tilde{\mathcal{N}}))$.
Therefore, our remaining task is to show that (the image in $K$-theory
of) the coherent sheaf restriction functor $i^{*}:K(Coh_{\mathcal{B}_{\chi}^{(1)}}(\tilde{\mathcal{N}}))\to K(Coh_{\mathcal{B}_{\chi}^{(1)}}(\tilde{S}_{\mathcal{N}}))$
is compatible under localization with the functor $\mbox{Wh}_{\chi}:K(Mod^{f.g.}(U^{\hat{\lambda}}(\mathfrak{g})_{\hat{\chi}})\to K(Mod^{f.g.}(U^{\hat{\lambda}}(\mathfrak{g},e)_{\hat{\chi}})$.
In fact, we shall show
\begin{lem}
Let $M\in Mod^{f.g.}(U^{\lambda}(\mathfrak{g})_{\chi})$, and let
$\mbox{Coh}(M)\in D^{b}(Coh_{\mathcal{B}_{\chi}^{(1)}}(\tilde{\mathcal{N}}))$
be its localization. Then there is an isomorphism 
\[
i^{*}\mbox{Coh}(M)\tilde{\to}\mbox{Coh}(\mbox{Wh}_{\chi}(M))
\]
 in $ $$D^{b}(Coh_{\mathcal{B}_{\chi}^{(1)}}(\tilde{S}_{\mathcal{N}}))$. \end{lem}
\begin{proof}
To prove this, we recall that the functor $\mbox{Coh}$ (for both
$U(\mathfrak{g})$ and $U(\mathfrak{g},e)$ modules) is obtained by
applying $\mathcal{L}^{\lambda}$ to arrive in $Mod(D^{\lambda})$
or $Mod(D(\chi)^{\lambda})$, then applying the splitting bundle for
$U(\mathfrak{g})$ or $U(\mathfrak{g},e)$. Further, this splitting
bundle, by definition, is compatible with the shift taking the weight
$\lambda$ to an unramified weight $\mu$. So we have 
\[
\mbox{Coh}(M)=M_{\chi}^{\mu}\otimes_{D^{\mu}}\tau_{\lambda}^{\mu}(D^{\lambda}\otimes_{U(\mathfrak{g})^{\lambda}}M)
\]
 where $\tau$ is the shift functor between weights and $M_{\chi}^{\mu}$
is the splitting bundle. Further, the completion of $D^{\mu}$ along
$\mathcal{B}_{\chi}^{(1)}$ is actually the pullback of the algebra
$U(\mathfrak{g})_{\hat{\chi}}^{\mu}$, and the splitting bundle $M_{\chi}^{\mu}$
is the pullback of the splitting bundle for $U(\mathfrak{g})_{\hat{\chi}}^{\mu}$
over the point $\chi^{(1)}\in\mathfrak{g}^{*,1}$. By abuse of notation,
we shall denote this splitting bundle also by $M_{\chi}^{\mu}$, so
that 
\[
\mbox{Coh}(M)=\pi^{*}(M_{\chi}^{\mu})\otimes_{D_{\hat{\chi}}^{\mu}}\tau_{\lambda}^{\mu}(D^{\lambda}\otimes_{U(\mathfrak{g})^{\lambda}}M)
\]
 Now, as explained above, the algebra $U(\mathfrak{g},e)_{\hat{\chi}}^{\mu}$
is an azumaya algebra over the completion of $\chi^{(1)}\in S_{\mathcal{N}}^{(1)}$.
It is related to $U(\mathfrak{g})_{\hat{\chi}}^{\lambda}$ as follows:
if $i:S_{\mathcal{N}}^{(1)}\to\mathcal{N}^{(1)}$ is the inclusion,
then $i^{*}U(\mathfrak{g})_{\hat{\chi}}^{\mu}$ is also an azumaya
algebra over the completion of $\chi^{(1)}\in S_{\mathcal{N}}^{(1)}$,
and hence is a matrix algebra (of dimension $p^{d(e)}$) over $U(\mathfrak{g},e)_{\hat{\chi}}^{\mu}$.
Further, the equivalence of categories 
\[
Mod(i^{*}U(\mathfrak{g})_{\hat{\chi}}^{\mu})\to Mod(U(\mathfrak{g},e)_{\hat{\chi}}^{\mu})
\]
 is given by a functor $V\otimes_{i^{*}U(\mathfrak{g})_{\hat{\chi}}^{\mu}}?$
for an appropriate $p^{d(e)}$-dimensional splitting bundle $V$.
By construction, this functor agrees with the functor $ $$\mbox{Wh}_{\chi}$
on $Mod(U(\mathfrak{g})_{\chi}^{\mu})$. Therefore, we have the compatibility
\[
i^{*}M_{\chi}^{\mu}\tilde{=}V\otimes E_{\chi}^{\mu}
\]
 where $E_{\chi}^{\mu}$ is the splitting bundle for the azumaya algebra
$U(\mathfrak{g},e)_{\hat{\chi}}^{\mu}$. So we have 
\[
i^{*}\mbox{Coh}(M)=i^{*}\pi^{*}(M_{\chi}^{\mu})\otimes_{i^{*}D_{\hat{\chi}}^{\mu}}i^{*}\tau_{\lambda}^{\mu}(D^{\lambda}\otimes_{U(\mathfrak{g})^{\lambda}}M)
\]
 Further, , the map $\pi\circ i$ conincides with the map $i\circ\pi|_{\tilde{S}_{\mathcal{N}}}$.
So 
\[
i^{*}\mbox{Coh}(M)\tilde{=}(\pi|_{\tilde{S}_{\mathcal{N}}})^{*}E_{\chi}^{\mu}\otimes(\pi|_{\tilde{S}_{\mathcal{N}}})^{*}(V)\otimes_{i^{*}D_{\hat{\chi}}^{\mu}}i^{*}\tau_{\lambda}^{\mu}(D^{\lambda}\otimes_{U(\mathfrak{g})^{\lambda}}M)
\]

To finish the argument, we note that one can play the same kind of
game as above with the localized algebras $D(\chi)^{\mu}$ and $D^{\mu}$-
namely that, after pulling back the above equivalences, we have that
$i^{*}D_{\hat{\chi}}^{\mu}$ is a matrix algebra over $D(\chi)_{\hat{\chi}}^{\mu}$,
and the equivalence between their module categories is given by $(\pi|_{\tilde{S}_{\mathcal{N}}})^{*}(V)\otimes_{i^{*}D_{\hat{\chi}}^{\mu}}?$.
Therefore, the sheaf 
\[
(\pi|_{\tilde{S}_{\mathcal{N}}})^{*}(V)\otimes_{i^{*}D_{\hat{\chi}}^{\mu}}i^{*}\tau_{\lambda}^{\mu}(D^{\lambda}\otimes_{U(\mathfrak{g})^{\lambda}}M)
\]
 is naturally a sheaf of $D(\chi)_{\hat{\chi}}^{\mu}$-modules, which
coincides with $\mathcal{L}^{\mu}(\mbox{Wh}_{\chi}(M))$ by all the
compatibilities of the splitting bundles. Therefore the sheaf $i^{*}\mbox{Coh}(M)$
is precicely $(\pi|_{\tilde{S}_{\mathcal{N}}})^{*}(E_{\chi}^{\mu})\otimes\mathcal{L}^{\mu}(\mbox{Wh}_{\chi}(M))$
which is $\mbox{Coh}(\mbox{Wh}_{\chi}(M))$ by definition.
\end{proof}

\section{Application To Modular Represetation Theory}

In this section, we shall present an application of the theory developed
in sections 1-6 to non-restricted representations of modular lie algebras.
In particular, we shall explain how the representation theory at a
regular $p$-character follows from what we have done. In particular,
we shall reprove a theorem of Jantzen and Soergel which describes
thie category $U_{\chi}^{\hat{0}}-mod$ for regular $\chi$. 

Thus we suppose, in this section and this section only, that $e$
is a regular nilpotent element of $\mathfrak{g}$. This implies that
the slice $S$ is the famous \emph{Kostant section}, which is isomorphic
to the scheme $\mathfrak{h}^{*}/W$ under the Chevalley map $\mathfrak{g}^{*}\to\mathfrak{h}^{*}/W$.\emph{
}In this case the resolution of the slice $\tilde{S}$ is isomorphic
to the scheme $\mathfrak{h}^{*}$, under the Grothendieck map $\tilde{\mathfrak{g}^{*}}\to\mathfrak{h}^{*}$
(this follows by pulling back the the first isomorphism). 

So now we would like to consider the sheaf $\tilde{D}(\chi)$. Since
there is already a map $U(\mathfrak{h})\tilde{=}O(\mathfrak{h}^{*})\to\Gamma(\tilde{D}(\chi))$,
which becomes the (comorphism of) $\tilde{\mathfrak{g}^{*}}\to\mathfrak{h}^{*}$
upon taking assoicated graded algebras, we deduce that in fact $\Gamma(\tilde{D}(\chi))=O(\mathfrak{h}^{*})$.
In fact, the $\Gamma$ is even unnecessary, since $\tilde{S}$ is
an affine variety; and we can say that $\tilde{D}(\chi)\tilde{\to}O(\mathfrak{h}^{*})$. 

So, by our localization for modular $W$-algebras, we have an equivalence
of categories: 
\[
D^{b}(Mod_{0}^{coh}(O(\mathfrak{h}^{*}))\to D^{b}(Mod_{0}^{f.g.}(U(\mathfrak{g},e)))
\]
 Next we want to specialize to the $p$-character $\chi$ . Instead
of the naive specialization of the above section, however, we shall
use the more refined version found in the work of Riche \cite{key-34}.
In particular, we shall need to consider the derived fibre of $\chi$
for the map $\tilde{S}\to S$. In general, this is a complicated object
which needs to be expressed in the language of derived algebraic geometry.
In our case, though, the map $\tilde{S}\to S$ is simply the quotient
$\mathfrak{h}^{*}\to\mathfrak{h}^{*}/W$, and we are considering the
fibre over the point $0$. This is a finite, flat morphism, and so
the dervied fibre is equal to the scheme theoretic fibre, which is
simply $\mbox{Spec}(A)$, where $A=\mbox{Sym}(\mathfrak{h})/\mbox{Sym}(\mathfrak{h})^{W}$. 

Therefore, we can state Riche's results in classical language. We
let $U(\mathfrak{g},e)_{\chi}$ denote the W-algebra at the strict
$p$-character $\chi$. Then we have 
\begin{thm}
Upon restricting to the derived fibre, the above equivalence becomes
an equivalence 
\[
D^{b}(Mod(A))\tilde{\to}D^{b}(Mod_{0}^{f.g.}(U(\mathfrak{g},e)_{\chi})
\]

\end{thm}
Let us relate this to the result of Jantzen and Soergel concerning
$U(\mathfrak{g})_{\chi}$. In particular, by Premet's theory, we have
a Morita equivalence 
\[
Mod(U(\mathfrak{g})_{\chi})\tilde{\to}Mod(U(\mathfrak{g},e)_{\chi})
\]
 which respects the action of the Harish-Chandra center. Therefore,
a regular block of $Mod(U(\mathfrak{g},e)_{\chi})$ is equivalent
to a regular block of $Mod(U(\mathfrak{g})_{\chi})$. The theorem
of Jantzen and Soergel describes the endomorphism algebra of a projective
generator of such a block, in particular, they prove that 
\[
End(P)\tilde{=}A
\]
 and therefore they deduce an equivalence of categories $Mod_{0}(U(\mathfrak{g})_{\chi})\tilde{\to}Mod(A)$. 

To deduce this (a priori) stronger result from the derived equivalence
above, we have to ask what happens to the standard $t$-structure
on $D^{b}(Mod_{0}^{f.g.}(U(\mathfrak{g},e)))$ (and on its completion
$D^{b}(Mod_{\hat{0}}^{f.g.}(U(\mathfrak{g},e)))$) under the first
equivalence above. In fact, this is a well-understood phenomenon,
due to the main result of \cite{key-6} and the comparison theorem
of section ? above. In particular, the resulting $t$-structure on
$D^{b}(Mod_{0}^{coh}(O(\mathfrak{h}^{*}))$ (and its completion $D^{b}(Mod_{\hat{0}}^{coh}(O(\mathfrak{h}^{*}))$)
is the exotic $t$-structure of \cite{key-6}. For our purposes here,
the important point is that the structure sheaf $O(\mathfrak{h}^{*})$
resricts to a projective object in the heart of this $t$-structure
on $D^{b}(Mod_{\hat{0}}^{coh}(O(\mathfrak{h}^{*}))$; and therefore,
after restricting to the derived fibre, the structure sheaf $A$ is
a projective object in the heart of the exotic $t$-structure on $D^{b}(Mod(A))$.
Therefore, in this case, the exotic $t$-structure coincides with
the usual one, and we recover the result of Jantzen and Soergel.

\curraddr{Department of Mathematics, Massachusetts Institute of Technology,
Cambridge, MA.}

\email{cdodd@math.mit.edu}
\end{document}